\documentclass[11pt]{amsart}
\usepackage{amssymb}
\usepackage{amsthm}
\usepackage{enumerate}
\usepackage{amsrefs}

\newcommand{\finiteG}{{\mathsf{G}}}
\newcommand{\finiteH}{{\mathsf{H}}}

\newcommand{\finiteC}{{\mathsf{C}}}
\newcommand{\finiteS}{{\mathsf{S}}}
\newcommand{\finiteT}{{\mathsf{T}}}
\newcommand{\finiteZ}{{\mathsf{Z}}}
\newcommand{\finiteM}{{\mathsf{M}}}
\newcommand{\finiteP}{{\mathsf{P}}}
\newcommand{\finiteB}{{\mathsf{B}}}

\newcommand{\pp}{{\mathfrak{p}}}
\DeclareMathOperator{\ind}{ind}

\DeclareMathOperator{\Int}{Int}

\DeclareMathOperator{\infl}{infl}
\DeclareMathOperator{\ext}{ext}

\DeclareMathOperator{\Hom}{Hom}
\DeclareMathOperator{\St}{St}
\DeclareMathOperator{\tr}{Tr}

\DeclareMathOperator{\stab}{stab}
\DeclareMathOperator{\JL}{JL}
\DeclareMathOperator{\Gal}{Gal}

\DeclareMathOperator{\BC}{BC}

\newcommand{\C}{{\mathbb{C}}}

\newcommand{\LL}{\mathcal{L}}
\newcommand{\OK}{\mathcal{O}}
\newcommand{\N}{{\mathcal{N}}}
\newcommand{\A}{\mathcal{A}}
\newcommand{\tA}{\widetilde{\A}}
\newcommand{\uG}{{\underline{G}}}
\newcommand{\tG}{{\widetilde{G}}}
\newcommand{\uC}{{\underline{C}}}
\newcommand{\uT}{{\underline{T}}}
\newcommand{\utG}{{\underline\tG}}
\newcommand{\uH}{{\underline{H}}}
\newcommand{\uJ}{{\underline{J}}}

\newcommand{\tH}{{\widetilde{H}}}
\newcommand{\tT}{{\widetilde{T}}}
\newcommand{\tC}{{\widetilde{C}}}
\newcommand{\tS}{{\widetilde{S}}}
\newcommand{\utH}{{\underline\tH}}
\newcommand{\utM}{{\underline\tM}}
\newcommand{\utC}{{\underline\tC}}
\newcommand{\utJ}{{\widetilde\uJ}}

\newcommand{\uM}{{\underline{M}}}
\newcommand{\tM}{{\widetilde{M}}}
\newcommand{\tB}{{\widetilde{B}}}
\newcommand{\tP}{{\widetilde{P}}}
\newcommand{\utP}{{\underline\tP}}
\newcommand{\tZ}{{\widetilde{Z}}}
\newcommand{\uZ}{{\underline{Z}}}
\newcommand{\uS}{{\underline{S}}}
\newcommand{\utZ}{{\underline{\tZ}}}

\newcommand{\ufG}{{\underline{\finiteG}}}
\newcommand{\tfG}{{\widetilde{\finiteG}}}
\newcommand{\tfB}{{\widetilde{\finiteB}}}
\newcommand{\utfG}{{\underline{\tfG}}}
\newcommand{\ufM}{{\underline{\finiteM}}}
\newcommand{\utfM}{{\widetilde{\underline{\finiteM}}}}
\newcommand{\utfS}{{\widetilde{\underline{\finiteS}}}}
\newcommand{\tfM}{{\widetilde{\finiteM}}}
\newcommand{\ufB}{{\underline{\finiteB}}}
\newcommand{\ufH}{{\underline{\finiteH}}}
\newcommand{\tfH}{{\widetilde{\finiteH}}}
\newcommand{\tfZ}{{\widetilde{\finiteZ}}}
\newcommand{\tfS}{{\widetilde{\finiteS}}}
\newcommand{\ufS}{{\underline{\finiteS}}}
\newcommand{\ufT}{{\underline{\finiteT}}}
\newcommand{\utfB}{{\underline{\widetilde\finiteB}}}
\newcommand{\utfT}{{\underline{\widetilde\finiteT}}}
\newcommand{\RTG}{R_\ufT^\ufG}

\newcommand{\B}{{\mathcal{B}}}
\newcommand{\D}{\mathcal{D}}
\newcommand{\F}{{\mathcal{F}}}
\newcommand{\reg}{^{\textup{reg}}}
\newcommand{\ereg}{^{\varepsilon\textup{-reg}}}
\newcommand{\conn}{^\circ}
 
\newcommand{\phm}{\phantom{-}}
\newcommand{\phprime}{\phantom{'}}
\newcommand{\inv}{^{-1}}
\newcommand{\lsup}[1]{{}^{#1}}
\newcommand{\mult}{^{\times}}
\newcommand{\phinv}{\phantom{{}\inv}}
\newcommand{\vphinv}{\vphantom{{}\inv}}

\newcommand{\E}{{\mathcal{E}}}

\newcommand{\reducible}[1]{(\ref{sec:depth-zero-G}PS--#1)}
\newcommand{\depthzerosc}[1]{(\ref{sec:depth-zero-G}SC--#1)}
\newcommand{\Hdepthzerosc}[1]{(\ref{sec:L-packets_u11}SC--#1)}
\newcommand{\cartan}[1]{(\ref{sec:cartan-G}--#1)}
\newcommand{\pslpacket}[1]{(\ref{sub:ps}--#1)}

\newcommand{\GL}{\textup{GL}}
\newcommand{\SL}{\textup{SL}}
\newcommand{\SO}{\textup{SO}}
\newcommand{\OO}{\textup{O}}
\newcommand{\PGL}{\textup{PGL}}
\newcommand{\UU}{U}

\newcommand{\set}[2]{
        {\left\{\left.
        #1\vphantom{#2\bigl(\bigr)}\,\right|
        \,#2\right\}}}

\newcommand{\smattwo}[4]{\left(
        \begin{smallmatrix}#1&#2\\#3&#4\end{smallmatrix}
        \right)}

\let\det\relax
\DeclareMathOperator{\det}{det}

\newtheorem{thm}{Theorem}[section] 

\newtheorem{prop}[thm]{Proposition} 
\newtheorem{lem}[thm]{Lemma} 
 
\theoremstyle{definition} 
 
\newtheorem{conj}[thm]{Conjecture} 
\newtheorem{hyp}[thm]{Hypothesis} 
\theoremstyle{remark}

\numberwithin{equation}{subsection}
 
\title{Depth-zero base change for ramified $U(2,1)$}
\date{\today}
\author{Jeffrey D.~Adler}
\address{Department of Mathematics and Statistics\\
American University\\
Washington, DC 20016-8050}
\email{jadler@american.edu}
\author{Joshua M.~Lansky} 
\address{Department of Mathematics and Statistics\\
American University\\
Washington, DC 20016-8050}
\email{lansky@american.edu}

\begin{document}

\begin{abstract}
  We give an explicit description of $L$-packets and quadratic base
  change for depth-zero representations of ramified unitary groups
  in two and three variables.  We show that this base
  change lifting is compatible with a certain lifting of families of
  representations of finite groups.
  We conjecture that such a compatibility is valid in much greater generality.
\end{abstract}

\maketitle


\section{Introduction}
\label{sec:intro}
Given a finite Galois extension $E/F$ of local fields
and a reductive algebraic $F$-group $\underline{G}$,
``base change'' is, roughly, a (sometimes only conjectural) 
mapping from (packets of)
representations 
of $G = \underline{G}(F)$
to those of $\underline{G}(E)$.
This mapping is essentially the same as the
Shintani lifting 
(as introduced in~\cite{shintani:sq-int}
and
extended
in~\cites{kawanaka:liftings-classical-1,
kawanaka:liftings-classical-2,
kawanaka:shintani,
gyoja:liftings,
digne:arcata,
digne:shintani
}
for finite groups)
in those cases in which both are known to exist
(see~\cites{langlands:gl2,%
flicker:gl3,%
arthur-clozel:trace,%
rogawski:u3,%
clozel:real_bc}).

Correspondences like base change that are associated to
the Langlands program can be difficult to describe explicitly,
even in cases where they are known to exist.
However, the Bernstein decomposition of the category of
smooth representations of $G$ gives rise to a partition
of the set of irreducible admissible representations,
and one hopes that base change is compatible with this
partition.
More specifically, given a Bushnell-Kutzko \emph{type}
\cite{bushnell-kutzko:smooth} that is conjecturally associated
to a piece of this partition, there should be a corresponding
type for a piece
of the corresponding partition of the representations of $\uG(E)$.
The work of
Bushnell and Henniart~\cites{%
bushnell-henniart:local-tame-1,%
bushnell-henniart:local-tame-2,%
bushnell-henniart:local-tame-4,%
bushnell-henniart:bc-glp%
}
provides evidence for such a compatibility when $\uG$ = $\GL(n)$;
and Blasco \cite{blasco:u3-bc-1} addresses the case of positive-depth
representations of $U(3)$.

Depth-zero representations are particularly interesting, since they
are associated to cuspidal representations of finite groups.
One can then ask whether base change for depth-zero representations
is compatible with base change for finite groups.
In the situation we consider in~\cite{adler-lansky:bc-u3-unram} and
the present article, there
is indeed such a compatibility.

To state this relationship, we briefly recall some notation.
Let $k_F$ and $k_E$ denote the residue fields of $F$ and $E$.
Given a point $x$ in the building of $\uG(F)$, one has
an associated parahoric subgroup $G_x \subseteq \uG(F)$.
This is a compact open subgroup, whose reductive quotient
is the group of rational points of a
connected reductive $k_F$-group $\ufG_x$
Thus, given a representation $\sigma$ of $\ufG_x(k_F)$,
one can inflate
to obtain a representation $\infl(\sigma)$ of $G_x$.
If $\utG$ is the restriction of scalars $R_{E/F}\uG$ of $\uG$,
then
one similarly has a parahoric subgroup $\tG_x \subseteq \utG = \uG(E)$, 
whose reductive quotient is $\widetilde\ufG_x(k_F)$.
If $E/F$ is unramified, then 
$\widetilde\ufG_x = R_{k_E/k_F} \ufG_x$,
so this last reductive quotient is equal to
$\ufG_x(k_E)$.

If $K$ is a compact open subgroup of $\uG(F)$ and $\rho$
is a smooth representation of $K$,
then we will say that a representation $\pi$ of $\uG(F)$ \emph{contains}
the pair $(K,\rho)$ if the restriction $\pi|_K$ contains $\rho$.
We will say that $\pi$ has \emph{depth zero}
if $\pi$ contains a pair of the form
$(G_x, \infl(\sigma))$, where $\sigma$
is an irreducible, cuspidal representation of $\ufG_x(k_F)$.
It makes sense to refer of \emph{depth-zero $L$-packets},
since one expects that either all or none of the representations
in a given $L$-packet will have depth zero.

\begin{conj}
\label{conj:unram}
Suppose $E/F$ is an unramified extension of local fields, 
and there is a base change lifting $\BC_{E/F}$ that takes $L$-packets
for $\uG(F)$ to $L$-packets for $\uG(E)$.
Let $\Pi$ denote a depth-zero $L$-packet for $\uG(F)$, and let $\pi\in \Pi$.
Suppose that $\pi$ contains $(G_x, \infl(\sigma))$, where $\sigma$
is an irreducible, cuspidal representation of
$\ufG_x(k_F)$.
Then some representation $\tilde\pi \in \BC_{E/F}(\Pi)$
contains
$(\tG_x,\infl(\BC_{k_E/k_F}(\sigma)))$,
where $\BC_{k_E/k_F}$ is the base change lifting
associated to the extension $k_E/k_F$
and the $k_F$-group $\ufG_x$.
\end{conj}

If $E/F$ is quadratic, and $\underline{G}$ is a unitary group
in two or three variables defined with respect to $E/F$, then
Rogawski~\cite{rogawski:u3} has shown that a base change lifting
exists, and has derived some of its properties.
The main theorem of \cite{adler-lansky:bc-u3-unram}
can be briefly stated as follows:
\begin{thm}
Conjecture \ref{conj:unram} is true when
$E/F$ is unramified quadratic,
and
$\uG$ is a unitary group in $2$ or $3$ variables defined
with respect to $E/F$.
\end{thm}
In order to prove this theorem, we had to make several assumptions.

We assumed $|k_F|>59$, but only to avoid some technical complications.
A weaker hypothesis is probably sufficient.
In order to be able to make use of Rogawski's construction of the
quadratic base
change lifting for $U(2,1)$, \cite{rogawski:u3},
we assumed that $F$ has characteristic zero.
We also needed to assume that base change for unitary groups
in two and three variables, as well as various endoscopic transfers,
preserve depth, at least for depth-zero $L$-packets.

In the present paper, we consider the case where $E/F$ is ramified.
Now the problem is completely different,
since we no longer have that
$\widetilde\ufG_x = R_{k_E/k_F} \ufG_x$.

\begin{conj}
\label{conj:main}
Suppose $E/F$ is a cyclic, tamely ramified extension of local fields,
and there is a
base change lifting $\BC_{E/F}$ that takes $L$-packets
for $\uG(F)$ to $L$-packets for $\uG(E)$.
Let $\Pi$ denote a depth-zero $L$-packet for $\uG(F)$, and let $\pi\in \Pi$.
Suppose that $\pi$ contains $(G_x, \infl(\sigma))$, where $\sigma$
is an irreducible, cuspidal representation of $\ufG_x(k_F)$.
Then some representation $\tilde\pi \in \BC_{E/F}(\Pi)$
contains
$(\tG_x,\infl(\tilde\sigma))$
for some irreducible representation $\tilde\sigma$ of $\tfG_x(k_F)$ contained
in $\ell(\sigma)$.
\end{conj}

Here, $\ell(\sigma)$ is a collection of representations of
$\widetilde\finiteG_x(k_F)$
that arises via a certain lifting (which we will later
call the \emph{$\varepsilon$-lifting}) of representations of finite
reductive groups.
When $E/F$ is unramified,
$\ell$ is compatible with the base change lifting
$\BC_{k_E/k_F}$.
The case where $E/F$ is ramified quadratic and $\uG$
is a unitary group in $2$ or $3$ variables
gives rise to liftings for several finite groups.  In the 3-variable
case, we describe these
liftings explicitly in \S\ref{sec:K-types}.
We prove basic properties of $\ell$ in general elsewhere
(see \cite{adler-lansky:lifting}).

Our main theorem is the following:

\begin{thm}
\label{thm:main}
Suppose Hypothesis \ref{hyp:depth}.
Conjecture \ref{conj:main} is true when $F$ has characteristic zero,
the order of $k_F$ is greater than $3$ and not equal to $7$,
$E/F$ is a totally and tamely ramified quadratic extension,
and $\uG$ is the quasisplit unitary group $U(2,1)_{E/F}$ associated to
$E/F$.
\end{thm}

However, we go somewhat farther, describing base change explicitly
for each depth-zero $L$-packet for $U(2,1)_{E/F}$
(see \S\ref{sec:base_change}).

As in \cite{adler-lansky:bc-u3-unram},
we assume that $F$ has characteristic zero only so that we
can apply results of Rogawski~\cite{rogawski:u3}.  Our calculations apply
equally well if $F$ is a function field of characteristic more than $3$.

In order to prove the theorem, we need to assume the following.

\begin{hyp}
\label{hyp:depth}
All representations within an $L$-packet have the same depth.
Considering unitary groups with respect to
a totally, tamely ramified quadratic extension $E/F$, 
the following endoscopic transfers,
defined with respect to a depth-zero character
(see \S4.8 of \cite{rogawski:u3}),
take depth-zero (resp.\ positive-depth) $L$-packets
to depth-zero (resp.\ positive-depth) $L$-packets.
\begin{itemize}
\item
from $(U(1,1)\times U(1))(F)$ to $U(2,1)(F)$;
\item
from $(U(1)\times U(1)\times U(1))(F)$
to $(U(1,1)\times U(1))(F)$.
\end{itemize}
\end{hyp}


Note that we do not include quadratic base change for $U(1,1)(F)$ or
$U(2,1)(F)$ (as defined in~\cite{rogawski:u3}) among the transfers
listed above.

We can leave out the quadratic base change lifting from $U(1,1)(F)$
to $\GL(2,E)$
because
we have just received a work of Blasco \cite{blasco:u11-bc}
that describes this base change.
In particular, Th\'eor\`eme 4.5 \emph{loc.\ cit.}\ implies
that depth is preserved in the above sense.


The quadratic base change lifting from $U(2,1)(F)$ to $\GL(3,E)$
is not included in the hypothesis because
all depth zero representations of $U(2,1)(F)$
are obtained via parabolic induction from a maximal torus or via transfer from
$(U(1,1)\times U(1))(F)$.
As a result, the base change lifts of depth-zero $L$-packets of
$U(2,1)(F)$ can be expressed in terms of the base change lifts of
associated packets of these simpler groups (see
Propositions~\ref{prop:bc_ps},~\ref{prop:bc_sc},
and~\ref{prop:bc_card_2}) for which depth-zero preservation is known.
(For unramified $U(2,1)$, this is not the case,
and therefore in~\cite{adler-lansky:bc-u3-unram}, the above assumption had
to be made for the base change lifting for $U(2,1)(F)$ as well.)

More generally, one expects that base change will multiply
depth by the ramification degree of $E/F$
(see~\cite{lansky-raghuram:depth}), but we do not need such a strong
statement here.


We expect that other correspondences of representations
can be made more explicit by examinations of compact open data.
For example, Silberger and Zink
\cite{silberger-zink:level-zero-matching}
have made the Abstract Matching
Theorem~\cites{dkv:amt,rogawski:amt,badulescu:jacquet-langlands}
explicit for depth-zero discrete series representations,
and Pan \cite{pan:theta-depth-zero} has shown that the theta
correspondence for depth-zero representations is compatible
with that for finite groups.

The structure of this article is as follows.  In
\S\ref{sec:notation-general}, we present our notation, as well as
some basic facts about Bruhat-Tits buildings and restriction of
scalars.   We fix a choice of nonarchimedean local field $F$.
In \S\ref{sec:liftings}, we review the general
notion of Shintani lifting for reductive groups over local and finite
fields.  Given a quasi-semisimple automorphism $\varepsilon$ of a finite
reductive group, we also describe the corresponding
$\varepsilon$-lifting of (families of) representations.
In \S\ref{sec:unitary-groups},
we fix a ramified quadratic extension $E/F$,
an associated ramified unitary group $\uG = U(2,1)$,
and a particular $F$-automorphism $\varepsilon$ of $\utG = R_{E/F}\uG$.
We discuss the structure of these groups.
We also
describe the reductive quotients of the stabilizer groups of points in
the Bruhat-Tits building of $\uG$ and classify their representations.
We then use this to describe the representations of $\uG(F)$ of depth
zero.  

In \S\ref{sec:u11},
we give an explicit description of the depth-zero $L$-packets
for ramified unitary groups in two variables,
as well as their base change lifts to $\GL(2,F)$.
There are two reasons to do so.
First, in order to describe explicit base change for all
representations of $U(2,1)$
(not just of depth zero), one would need to understand depth-zero base change
for these smaller groups, since they arise in
the constructions by J.-K. Yu \cite{yu:supercuspidal}
and S. Stevens \cite{stevens:classical-sc}
of positive-depth
supercuspidal representations.
Second, some of our understanding of base change
for depth-zero representations comes via endoscopy,
and $U(1,1) \times U(1)$
is an endoscopic group for $U(2,1)$.
In order to handle the case of the anisotropic group $U(2)$, we
must assume the existence of a Jacquet-Langlands-like correspondence
between depth-zero, discrete-series representations of $U(1,1)$
and those of $U(2)$.
Conjecturally, such a correspondence
is determined by the theta correspondence between $U(2)$ and $U(1,1)$.
See \S\ref{sec:bc-u2}.

In \S\ref{sec:u21}, we explicitly determine the
depth-zero $L$-packets and $A$-packets for $\uG(F)$.
Certain $A$-packets and $L$-packets of $G$ have size $2$ and contain both a
supercuspidal and a non-supercuspidal representation.
Such packets
are parametrized by certain (Weyl equivalence classes of) characters $\lambda$
of the diagonal torus of $U(2,1)(F)$.
Proposition~\ref{prop:A-packet}
completely determines the relationship between the two representations
in such packets, and this is sufficient for the verification of
this case of Theorem~\ref{thm:main} in \S\ref{sec:base_change}.
In particular, an understanding of the explicit dependence of the elements
of such a packet on the character $\lambda$ is not required.
For completeness, we provide such a description anyway
in Proposition~\ref{prop:A-packet_explicit}.
The proof, which we do not include here,
involves computations
of the
values of certain irreducible characters of $\SL(2,k_F)$ on unipotent
elements.

In
\S\ref{sec:base_change}, we determine the base change lift of each of
the depth-zero packets of $\uG(F)$.
In \S\ref{sec:K-types} we verify that 
the types of the representations in a given packet are related
to those of the base change lift of the packet according to
Theorem~\ref{thm:main}.

The authors thank James Arthur, Martin H. Weissman,
and Wee-Teck Gan
for helpful conversations;
an anonymous referee for useful comments;
and the National Security Agency (H98230-07-1-002)
and the National Science Foundation (DMS-0854844)
for partial support.

\section{General notation and facts}
\label{sec:notation-general}
Given a field $K$, we will denote its algebraic closure by $\overline K$.
We will generally use underlined letters to denote
algebraic groups, and corresponding ordinary letters to denote
groups of rational points, where the field of rationality will be clear
from the context.
For any algebraic $K$-group $\uH$, we will let $\uH\conn$ denote the
connected component.

We will say that an element
$h$ of $\uH$ is \emph{$\varepsilon$-semisimple}
if $h\varepsilon$ is a semisimple
element of the semidirect product
$\uH\rtimes\langle\varepsilon\rangle$.
Let $h$ be such an element, and let $\uC$ be the fixed point group of
$\Int(h)\circ\varepsilon$ in $\uH$.  Then we will call $h\in\uH$
\emph{$\varepsilon$-regular} (resp.~\emph{strongly $\varepsilon$-regular})
if $\uC\conn$ (resp.~$\uC$) is a torus.
If $\varepsilon$
is trivial we will drop the $\varepsilon$ in the above terminology.
Note that $h$ is $\varepsilon$-regular if and only if $h\epsilon$ is a
regular element of $\uH\rtimes\langle\varepsilon\rangle$.
Let $\uH\ereg$ denote
the set of $\varepsilon$-regular elements of $\tH$.  If $\varepsilon$
is trivial, this is just the set $\uH\reg$ of regular elements of $\uH$.

Given a finite-order $K$-automorphism $\varepsilon$ of an algebraic
$K$-group $\uH$, we will say that two strongly regular elements $g,h\in H$
are \emph{$\varepsilon$-conjugate} (resp.~\emph{stably
  $\varepsilon$-conjugate}) if
$h=xg\varepsilon(x)^{-1}$ for some $x$ in $H$
(resp.~$\uH(\overline K)$).  

Let $\uG$ be an algebraic $K$-group, and let $L/K$ be a finite extension.
Let $\utG =R_{L/K}(\uG)$, the algebraic $K$-group obtained from
$\uG$ via restriction of scalars from $L$ to $K$.
Then $\utG(K)$ can be identified with
$\uG(L)$.  Moreover, $\utG$ is $L$-isomorphic to the direct
product $\uG\times\cdots\times\uG$ of $[L:K]$ copies of $\uG$.  We
will be interested in the case where $L/K$ is
cyclic.  In this case, let $\iota$ be a generator of $\Gal(L/K)$.
Then  $\utG$ is equipped with a $K$-automorphism $\varepsilon$ whose action on
$\utG(K)$ corresponds to that of $\iota$ on $\uG(L)$.  Moreover, the
fixed-point
group $\utG^\varepsilon$ can be identified with $\uG$.

If $L/K$ is a quadratic extension,
we let $\UU (1,L/K)$ denote
the kernel of the norm map $R_{L/K}(\GL(1))\rightarrow\GL(1)$ of
$K$-groups given by $x\mapsto x\varepsilon (x)$, where $\varepsilon$
is the automorphism of $R_{L/K}(\GL(1))$ corresponding to the
nontrivial element of $\Gal(L/K)$.
We will sometimes denote $\UU(1,L/K)(K)$ by $L^1$ when
$K$ is understood.

For any reductive algebraic group $\uG$ defined
over either a finite or nonarchimedean local field $K$,
we have the following notation and terminology.  (Here
we use $G$ to denote the group of rational points of $\uG$.)
\begin{itemize}
\item
$\E(G)$ will denote the set of irreducible representations of $G$.
\item
$\mathbf{1}_G$ will denote the trivial representation of $G$.
\item
$\St_G$ will denote the Steinberg representation of $G$.
\item For any representation $\sigma$ of a subgroup $H$ of $G$,
$\ind_H^G\sigma$ will denote the representation of $G$ obtained from
$\sigma$ via normalized compact induction.
If $\sigma$ is irreducible, then $\omega_\sigma$ will denote
the central character of $\sigma$.
\item
For any admissible, finite-length representation
$\pi$ of $G$,
let $\theta_\pi$ denote the character of $\pi$,
considered either as a function on the set of elements or conjugacy classes
of $G$ (of $G\reg$ in the local-field case),
or as a distribution
on an appropriate function space on $G$.
\item
Suppose $\varepsilon$ is an automorphism of $G$.
Then
$\varepsilon$ acts in a natural way on the set of equivalence
classes of irreducible, admissible representations of $G$.
Suppose
$\pi$ is such a representation and $\pi\cong \lsup\varepsilon\pi$.
Let
$\pi(\varepsilon)$ denote an intertwining operator from $\pi$ to
$\lsup\varepsilon\pi$.
If $\varepsilon$ has order $\ell$, then we can
and will normalize $\pi(\varepsilon)$ by requiring that the scalar
$\pi(\varepsilon)^\ell$ equal $1$.
Then $\pi(\varepsilon)$ is well
determined up to a scalar $\ell$th root of unity.
The
\emph{$\varepsilon$-twisted character} of $\pi$ is the distribution
$\theta_{\pi,\varepsilon}$ defined by
$\theta_{\pi,\varepsilon}(f) = \tr(\pi(f)\pi(\varepsilon))$
for $f\in C_c^\infty(G)$,
the space of compactly supported, locally constant functions on $G$.
As
with the character, the twisted character can be represented by a
function (again denoted $\theta_{\pi,\varepsilon}$) on $G$ ($G\ereg$
in the local-field case), or on its set of
$\varepsilon$-conjugacy classes.

\item 
A class function $f:G\rightarrow \C$ will be called \emph{stable} if
it is constant on every stable conjugacy class of $G$.

\item
A locally constant function $f$ on $G$ will be called an
\emph{$\varepsilon$-stable class function},
or simply \emph{$\varepsilon$-stable}, if
$f(x)=f(y)$ for all $x,y\in G$ such that $y = gx\varepsilon (g)^{-1}$
for some $g\in\uG(\overline F)$.

\item
For any maximal $K$-torus $\uT$ of a $K$-group $\uG$,
let $W(T,G)$ denote the quotient of $T$ in its normalizer in $G$, and
let $W_K(\uT,\uG)$ denote the group of $K$-points of the absolute
Weyl group $N_{\uG}(\uT)/\uT$.
\end{itemize}

Now fix a nonarchimedean local field $F$ of residual characteristic $p$.
For any algebraic extension $E/F$, let $\OK_E$ denote
the ring of integers of $E$, $\mathfrak{p}_E$ the prime ideal
in $\OK_E$, and $k_E= \OK_E/\mathfrak{p}_E$ the residue field.
Let $q_E$ denote the order of $k_E$.
Let $|\cdot |_E$ denote the absolute value on $E$,
normalized so that primes in $\OK_E$ have
absolute value $q_E\inv$.  In the case $E=F$, we will drop the
subscripts $E$ above. 
For any abelian extension $E/F$, let
$\omega_{E/F}$ denote the character of $F\mult$
associated to $E/F$ by local class
field theory.

Let $\widehat{\uG}$ denote the complex dual group of $\uG$.  The
$L$-group of $\uG$ is the group $\lsup L\uG = \widehat\uG\rtimes W_F$,
where $W_F$ is the Weil group of $F$.
We note that if $\utG = R_{E/F}\tG$, then
the dual group of $\utG$ is the direct product of $[E:F]$
copies of $\widehat\uG$,
and we have a natural diagonal embedding of $\lsup L \uG$ in $\lsup
L\utG$.

For every reductive algebraic
$F$-group $\uG$,
and every algebraic extension $E/F$ of finite ramification degree,
one has an associated
\emph{extended affine building} $\B(\uG,E)$,
as defined by Bruhat and
Tits~\cites{bruhat-tits:reductive-groups-1,bruhat-tits:reductive-groups-2}.
As a $\uG(E)$-set,
$\B(\uG,E)$ is a direct product of an affine space (on which
$\uG\conn(E)$ acts via translation) and the \emph{reduced building}
$\B^{\text{red}}(\uG,E)$, which depends only on
$\uG/\uZ$, where $\uZ$ is the center of
$\uG$.
Note that $\uZ(E)$ fixes
$\B^{\text{red}}(\uG,E)$.
The building $\B(\uG,F)$ always has a natural embedding
into $\B(\uG,E)$.
To every point $x\in\B(\uG,E)$,
there is an associated
parahoric subgroup $\uG(E)_x$
of $\uG(E)$.
The stabilizer $\stab_{\uG(E)}(x)$ of $x$ in $\uG(E)$
contains $\uG(E)_x$ with finite index.

The pro-$p$-radical of $G_x$ is denoted $G_{x+}$,
and the quotient $\stab_G(x)/G_{x+}$ is the group
$\finiteG_x$ of
rational points of a reductive $k_F$-group
$\ufG_x$.
The quotient $G_x/G_{x+}$ is the group
$\finiteG_x\conn$ of
rational points of $(\ufG_x)\conn$.
(Note that it is more common to use $\ufG_x$
to denote what we are calling $\ufG_x\conn$.)

More generally, one could deal with parahoric subgroups of $\uG(E)$,
and could form the quotient
$\stab_{\uG(E)}(x)/\uG(E)_{x+}$,
which is isomorphic to the group of rational points of
a reductive $k_E$-group that we will for the moment denote $\ufG_x^E$
(we will not need this notation later).
Note that if $E/F$ is unramified,
then $\ufG_x^E=R_{k_E/k_F}\ufG_x$,
so that $\finiteG_x^E=\ufG_x(k_E)$.
Moreover, we may identify $\Gamma$ with $\Gal(k_E/k_F)$,
and the resulting actions of $\Gamma$ on
$\finiteG_x^E$ and on $\ufG_x(k_E)$
are the same.
But when $E/F$ is ramified, the situation is more complicated.
See \S\ref{sub:reductive-quotients} for examples.

These objects do not depend specifically on $x$.
Rather, they depend on the facet $\F$ in $\B(\uG,E)$ containing $x$.
Thus, we will feel free to replace $x$ by $\F$ in the subscripts above.
In case $\uG\conn$ is a torus,
the building contains only one facet, and
we may write $G_0$, $G_{0+}$, and
$\ufG$ instead of $G_x$, $G_{x+}$, and
$\ufG_x$, since the latter are independent
of the choice of $x$.

In general, given an $E$-group and a point or facet in
the building, we will use sans-serif lettering to denote the
corresponding $k_E$-group.

Suppose $E/F$ is a Galois extension, and let $\Gamma=\Gal(E/F)$.
If $E/F$ is tamely ramified, then by Corollary~2.3
of \cite{yu:supercuspidal},
the group of fixed points $(\finiteG_x^E)^\Gamma$
for the action of $\Gamma$ on $\finiteG_x^E$ is equal to
$\finiteG_x$.

Let $\uT$ be an $F$-torus.  If a character $\lambda$ of $T$ is
trivial on $T_{0+}$, we let $\bar\lambda$ denote the corresponding
character of $\finiteT = T_0/T_{0+}$.

Following the notation in~\cite{rogawski:u3}, if $\uT$ is an $F$-torus
in $\uG$, let
$\D_\uG(\uT/F)$ denote the kernel of the natural map
$H^1(F,\uT)\longrightarrow H^1(F,\uG)$.

If $\ufG$ is a connected reductive $k_F$-group,
$\ufT$ is a maximal $k$-torus in $\ufG$,
and $\theta$ is a (complex) character of $\finiteT$, then
let 
$\RTG \theta$
denote the corresponding Deligne-Lusztig virtual character
of $\finiteG$~\cite{deligne-lusztig:finite}.

Given a surjection $A \longrightarrow B$ between abstract groups
$A$ and $B$, and a representation $\rho$ of $B$, we can pull back
(or ``inflate'') $\rho$ to form a representation of $A$, denoted
$\infl(\rho)$.  Whenever this notation is used,
the surjection will be clear from the context.

\section{Liftings}
\label{sec:liftings}
\subsection{Norm mappings and Shintani lifting}
\label{sub:shintani}
Let $\uH$ be a connected reductive group defined over a field $K$, and let
$\varepsilon$ be a $K$-automorphism of $\uH$ of finite order $\ell$.
One can define a norm mapping $N_\varepsilon^\uH$ from $\uH$ 
to $\uH$ by
$$
h\mapsto h \cdot \varepsilon(h) \cdot \cdots \cdot \varepsilon^{\ell-1}(h).
$$
We note that if elements $h$ and $h'$ in $\uH (\overline K)$
are $\varepsilon$-conjugate, then
$N_\varepsilon^\uH(h)$ and $N_\varepsilon^\uH(h')$ are conjugate in $\uH
(\overline K)$.
Thus the norm gives a map from the set of
$\varepsilon$-conjugacy classes in $\uH$ to the set of conjugacy
classes in $\uH$.

Suppose $\uH$ is abelian, and let $\uC = (\uH^\varepsilon)^\circ$  be the
identity component of the group of $\varepsilon$-fixed points of $\uH$.
Then $N_\varepsilon^\uH$ gives an $F$-homomorphism from $\uH$ to $\uC$.
For any character $\lambda$ of $C$, define the lifted character
$\tilde\lambda$ of $H$ by $\tilde\lambda =
\lambda\circ N_\varepsilon^\uH$.  Note the abuse of notation in this
definition---the group $\uC$ does not,
of course, determine $\uH$ or $\varepsilon$.
However, whenever we use this
notation, $\uH$ and $\varepsilon$ will be clear from the context.

Now suppose $L/K$ is a finite cyclic extension of degree $\ell$ of
local or finite fields.  Let
$\iota$ be a generator of $\Gal (L/K)$.
As in~\S\ref{sec:notation-general}, let
$\uG$ be a connected reductive $K$-group and $\utG=R_{L/K}\uG$. Let
$\varepsilon$ be the $K$-automorphism of $\utG$ corresponding to
$\iota$.
%
Under various conditions
(for example,
if $\uG$ is quasi-split and has a simply connected derived group),
the map $N_\varepsilon^\utG$
determines an injective map
$\N = \N^{\uG}_{L/K}$ from the set of
stable, strongly $\varepsilon$-regular $\varepsilon$-conjugacy classes
of $\tG$ to the set of stable, strongly regular conjugacy classes in $G$.

If $\Pi$ and $\widetilde\Pi$ are finite sets of
irreducible (smooth if $K$ is local) representations
of $G$ and $\tG$, respectively,
we say that $\widetilde{\Pi}$
is the \emph{Shintani lifting}
$\Pi$ if
$$
\Theta_{\widetilde{\Pi},\varepsilon}(g) = 
\Theta_\Pi(\N(g))
$$
for all $g\in \tG\ereg$ (all $g\in \tG$ in the
finite-field case), where $\Theta_\Pi$ and
$\Theta_{\widetilde{\Pi},\varepsilon}$ are nontrivial stable
(resp.~$\varepsilon$-stable) linear combinations of the characters
(resp.~$\varepsilon$-twisted characters) of the elements of $\Pi$
(resp.~$\widetilde{\Pi}$).
Note that, at least in the $p$-adic case,
it is easily seen that the Shintani lifting is independent of the choice
of generator $\iota$.

In most cases, one expects the Shintani
and base change liftings to coincide, although exceptions to this are
known (see~\cite{jkim-ps:theta10} and \cite{rogawski:u3}*{\S13.2}).
Rogawski has shown that this is indeed the case for all irreducible
smooth tempered representations of quasisplit $p$-adic unitary groups
in two and three variables~\cite{rogawski:u3}*{\S11,13}.

\subsection{A related lifting for representations of finite groups of Lie type}
\label{sub:lifting-finite}
For a connected reductive $k$-group $\ufH$,
let $\ufH^*$ denote the dual $k$-group.
The dual group is not canonical, but no statement we make will depend
on its explicit realization.  There is a canonical one-to-one
correspondence between the rational conjugacy classes of maximal $k$-tori of
$\ufH$ and those of $\ufH^*$~\cite{carter:finite}*{Prop.~4.3.4}.  Let
$\ufT$ be a maximal $k$-torus of $\ufH$ and let $\ufT^*$ be a
maximal $k$-torus of $\ufH^*$ whose rational conjugacy class
corresponds to that of $\ufT$.  We will say that $\ufT$ and $\ufT^*$
are {\em $k$-dual} to each other.  Once a fixed embedding of
$k\mult$ in $\C\mult$ is chosen, we may identify $\finiteT^*$ with the
group of characters of $\finiteT$.

By work of Lusztig \cite{lusztig:disconnected-centre} and
Digne-Lehrer-Michel \cite{digne-lehrer-michel:nonconnected-center},
the set $\E(\finiteH)$ of irreducible representations of $\finiteH$ can be
partitioned into a collection of families $\E_{(s)}(\finiteH)$, called
\emph{rational series}, parametrized by
semisimple conjugacy classes $(s)$ in $\finiteH^*$.
The set $\E_{(s)}(\finiteH)$ consists of those elements of $\E(\finiteH)$
whose characters occur in a Deligne-Lusztig virtual character
$R_\ufT^\ufH \chi$, where $\ufT$ is a maximal $k$-torus of $\ufH$ that
is $k$-dual to a maximal
$k$-torus $\ufT^*$ of $\ufH^*$ containing $s$,
and $\chi$ is the character of $\finiteT$ that corresponds to
$s\in \finiteT^*$.

Now, as in \S\ref{sub:shintani},
let $\utfG$ denote a connected
reductive $k$-group, and $\varepsilon$ a $k$-automorphism
of $\utfG$.
We will assume that $\varepsilon$ is \emph{quasi-semisimple},
meaning that $\varepsilon$ fixes some pair $(\utfB_0,\utfT_0)$ consisting
of a Borel subgroup $\utfB_0$ of $\utfG$
and a maximal torus $\utfT_0$ of $\utfB_0$.
Let $\ufG = (\utfG^\varepsilon)\conn$
denote the connected part of the group of
$\varepsilon$-fixed points in $\utfG$.
From \cite{steinberg:endomorphisms}*{Corollary 9.4}
or \cite{kottwitz-shelstad:twisted-endoscopy}*{Theorem 1.1.A(1)},
$\ufG$ is reductive.
Let $\ufT \subseteq \ufG$ denote a maximal $k$-torus, and let
$\utfT$ denote the centralizer of $\ufT$ in $\utfG$.
Then from \cite{kottwitz-shelstad:twisted-endoscopy}*{Thm.~1.1.A},
$\utfT$ is a maximal $k$-torus in
$\utfG$.

The maximal $k$-tori $\ufT$ and $\utfT$ determine, up to $k$-conjugacy,
maximal $k$-tori $\ufT^*$ and $\utfT^*$ in the dual groups
$\ufG^*$ and $\utfG^*$, respectively.
The map $N^{\utfT}_\varepsilon\colon \utfT \longrightarrow \ufT$
determines a map
$N^*_{\ufT^*,\utfT^*}\colon \ufT^* \longrightarrow \utfT^*$.

\begin{thm}
\label{thm:epsilon_lifting}
There is a map
\begin{multline*}
\N^* \colon
\{\text{semisimple conjugacy classes in $\finiteG^*$}\} \\
\longrightarrow
\{\text{semisimple conjugacy classes in $\tfG^*$}\}
\end{multline*}
that is compatible with the maps $N^*_{\ufT^*,\utfT^*}$
in the sense that if $s\in \finiteT^*$, then
$N^*_{\ufT^*,\utfT^*}(s) \in \N^*((s))$.
\end{thm}

Thus, we have a lifting of families of irreducible representations
from $\finiteG$ to $\tfG$:
$$
\E_{(s)}(\finiteG) \mapsto
\E_{\N^*((s))}(\tfG).
$$
We will refer to this as the \emph{$\varepsilon$-lifting.}

We will prove Theorem~\ref{thm:epsilon_lifting} elsewhere
\cite{adler-lansky:lifting}.
We will also refine the $\varepsilon$-lifting by showing in many cases
one can lift subsets
of $\E_{(s)}(\finiteG)$
to subsets of
$\E_{\N^*((s))}(\finiteG)$.
For now, note that
if the character of an irreducible representation $\pi$ of $\finiteG$
appears in
$R_\ufT^\ufG \chi$,
then the character of the $\varepsilon$-lifting $\tilde\pi$ of $\pi$
is contained in
$R_\utfT^\utfG \tilde\chi$,
where $\tilde\chi = \chi \circ N^{\utfT}_\varepsilon$.

\section{Unitary groups in three variables}
\label{sec:unitary-groups}
\subsection{Specific notation}
\label{sec:notation-special}
We will assume from now on that the characteristic of the
nonarchimedean local field $F$ is zero and
the residual characteristic $p$ is odd.
Let $\overline{F}$ denote a fixed algebraic closure of $F$.
Let $\varpi$ denote a uniformizer of $F$.
Let $E$ and $E'$ denote the two ramified quadratic
extensions of $F$ in $\overline{F}$.
Then $k_E = k_{E'} = k$.
Fix uniformizers $\varpi_E$ and $\varpi_{E'}$ for $E$
and $E'$ (respectively)
such that
$\varpi_E^2 = \varpi_{E'}^2 = \varpi$.
Also, fix a character $\Omega$ of $E\mult$ extending $\omega_{E/F}$
that is trivial on $1+\mathfrak{p}_E$.  Note that there are two such
characters.  Let $\iota$ denote the generator of $\Gamma :=
\Gal(E/F)$.  For any extension $L/F$, we have an action of $\iota$ on
$E\otimes L$ given by $x\otimes y\mapsto \iota(x)\otimes y$.

Let $\varepsilon'$ be the $F$ automorphism of $R_{E/F}\GL(3)$
corresponding to $\iota$.  Let
$$
\Phi =
\left(
\begin{smallmatrix}
0&\phm0&\phm 1\\
0&-1&\phm 0\\
1&\phm0&\phm 0
\end{smallmatrix}
\right)\in\GL(3,F)\subset \left(R_{E/F}\GL(3)\right)(F) .
$$
From now on, let $\uG$ denote the group of fixed points in
$R_{E/F}\GL(3)$ of the automorphism $\varepsilon : x\mapsto
   \Int(\Phi)(\lsup{t}\varepsilon'(x)\inv)$, where $\lsup{t}y$ denotes
the transpose of the matrix $y$.
Then $\uG$ is the unique (up to isomorphism) unitary group $U(2,1)$ in three
variables over $E/F$, and $\utG$ is isomorphic to $R_{E/F}\GL(3)$.  For
any extension $L/F$, we can and will explicitly realize $\uG(L)$ as
the group
$$
\set{ g\in\GL(3,E\otimes L) }{ g\, \Phi\ \lsup{t}(\iota(g)) = \Phi},
$$
where the action of $\iota$ on $\GL(3,E\otimes L)$ is coordinatewise.
Then the action of $\varepsilon$ on $\GL(3,E\otimes L)$ is given by
$g\mapsto\Phi\lsup{t}\iota(g)\inv\Phi\inv$.

Let $\uZ$ denote the center of $\uG$.
So, following our notational conventions,
$\underline{\widetilde{Z}}$ can be identified with the center of $\utG$,
and $\finiteZ$ is the reductive quotient of $Z$.
Note that, since $Z$ is isomorphic to $E^1$,
$\finiteZ$ has only two elements.

Let $\B = \B(\uG,F)$ and
$\widetilde{\B} = \B(\utG,F)$.  Then $\widetilde\B$ can be identified
with $\B(\uG,E)$, and the action of $\varepsilon$ on $\widetilde\B$
corresponds to the action of $\iota$ on $\B(\uG,E)$.
Thus from \cite{rousseau:thesis} we may and
will identify the set of fixed points $\widetilde{\B}^\varepsilon$ with $\B$.

For any $F$-group $\underline{Q}$ for which the norm map is defined,
let $\N^{\underline{Q}}=\N^{\underline{Q}}_{E/F}$.
When $\underline{Q} = \uG$, we simply write $\N$.

There is a unique homomorphism $\utG\rightarrow\GL(1)$ that agrees with the
usual determinant map on $\tG$.  Given any algebraic subgroup
$\underline S$ of $\utG$, we will denote by $\det_{\underline S}$ the
restriction of this homomorphism to $S$.  
We will omit the
subscript when it is clear from the context.

Since $\uG$ is $F$-quasisplit, it contains $F$-Borel
subgroups.
In particular, $\widetilde\B$ must contain some $\varepsilon$-invariant
apartment $\tA$ with more than one $\varepsilon$-fixed point.
Choose an $\varepsilon$-fixed point $y$ in an $\varepsilon$-invariant
minimal facet in $\tA$
and an $\varepsilon$-invariant alcove $\widetilde\F$ in $\tA$
such that the closure of $\widetilde\F$ contains $y$.
Let $\F$ denote the set of $\varepsilon$-fixed points of $\widetilde\F$.
Then these choices determine an $F$-Borel subgroup $\underline{B}$
together with a Levi factor $\uM$ of $\underline{B}$.
Note that $M$ is isomorphic to $E\mult \times E^1$.
We may assume that our choices of $y$ and $\widetilde\F$ are such that
$B$ is the group of upper-triangular matrices in $G$,
$M$ is the group of diagonal matrices, and $\tG_y$ is $\GL(3,\OK_E)$.

The boundary of $\F$ contains two points:
the previously chosen point $y$, and another point that we will
denote $z$.
Note that
$\widetilde\F$ is the direct product of a one-dimensional affine space
and an $\varepsilon$-invariant equilateral
triangle $\Delta$ 
in the reduced building of $\utG$
(which we will identify with a subset of $\widetilde\B$),
$y$ is the $\varepsilon$-fixed vertex of $\Delta$,
and $z$ is the midpoint of the wall of $\Delta$ that is opposite $y$.
In $\B$, $y$ and $z$ are both special vertices,
but neither is hyperspecial.

Let $\ufB_y$ (resp.~$\ufB_z$) denote the Borel
subgroup of $\ufG_y$ (resp.~$\ufG_z$)
determined by $\F$.

Let $\utH$ be the subgroup of $\utG$ whose $L$-rational points (for
any extension $L/F$) consist
of those matrices of the form
$$
\left(
\begin{smallmatrix}
* & 0 & * \\
0 & * & 0 \\
* & 0 & *
\end{smallmatrix}
\right) \subset \GL(3,E\otimes L).
$$
Then $\utH$ is a Levi factor of a parabolic $F$-subgroup $\utP$.
Let $\uH$ denote the group $\utG\cap\uG$.
Note that $\utH$ is isomorphic to $R_{E/F}\uH$.
We have that $\uH$ is an
$E$-Levi, but not $F$-Levi, subgroup of $\uG$.
It is an
endoscopic group for $\uG$, isomorphic to a direct product of
$\uZ\cong\UU(1,E/F)$ and the $F$-subgroup $\uH^0$ of
$\uH$ consisting of matrices of the form 
$$
\left(
\begin{smallmatrix}
* & 0 & * \\
0 & 1 & 0 \\
* & 0 & *
\end{smallmatrix}
\right).
$$
We note that $\uH^0$ is a quasi-split ramified unitary group
in two variables.

\subsection{Reductive quotients of stabilizers and their representations}
\label{sub:reductive-quotients}
For us, $\widetilde{\ufG}_x$ will always
mean the reductive quotient of $\tG$ corresponding
to the point $x\in\B(\uG,F)$, rather than, say,
the restriction of scalars
$R_{k_E/k_F} \ufG_x$,
though these are the same when $E/F$ is unramified.
For $x$ equal to either of the points $y$ or $z$,
or the facet $\F$ in between them,
we will compute $\finiteG_x$ and describe its irreducible representations.

\subsubsection{The group $\finiteG_y$} 
\label{sub:finiteG_y}
Recall that $\tG_y$ is the group $\GL(3,\OK_E)$.
Explicitly, $\tG_y$ is the stabilizer in $\GL(3,E)$ of the lattice
$\LL= \OK_E \oplus \OK_E \oplus \OK_E$ inside $V=E^3$.
Then $\tG_{y+}$ is the subgroup of elements whose image
in $\GL(\LL / \varpi_E \LL)$ is trivial.
Thus, with our identifications,
$\widetilde{\finiteG}_y = \GL(\LL / \varpi_E \LL)
\cong \GL(3,k)$.
The action of $\varepsilon$ is given by
$g \mapsto \Phi (\lsup{t} g \inv) \Phi\inv$,
so
$\finiteG_y = \OO(2,1)(k)$,
the orthogonal group determined by the symmetric form
whose matrix with respect to the standard basis is $\Phi$,
and $\finiteG_y\conn = \SO(2,1)(k)$.
In particular, $\finiteG_y = \finiteG_y\conn \times \finiteZ$.

We now briefly classify the irreducible representations of
$\finiteG_y$.
Since any such representation is of the form
$\tau\otimes\nu$ for some representation $\tau$ of $\finiteG_y\conn$
and one of the two characters $\nu$ of $\finiteZ$,
we turn our attention to the
irreducible representations of $\finiteG_y\conn$.
Since $\SO(2,1)(k)$ is isomorphic to $\PGL(2,k)$,
representations of $\finiteG_y\conn$ correspond to representations
of $\GL(2,k)$ having trivial central character.
The following classification is easily derived from
\cite{srinivasan:finite-book}*{Chapter VIII}.

The group $\ufM\conn$ embeds in the Borel subgroup $\ufB_y$ of
$\ufG_y\conn $.  If $\chi$ is a character of $\finiteM\conn$, then the
induced representation
$\ind_\finiteB^{\finiteG_y\conn }\chi$ is irreducible except when
$\chi$ extends to a character $\chi_0$ of $\finiteG_y\conn $
(i.e., when $\chi$ has order dividing 2).
In this case, the induced
representation is the sum of $\chi_0$ and
$\St_{\finiteG_y\conn }\cdot\chi_0$.

The remaining representations are cuspidal of the following form.
If $l/k$ is
a quadratic extension, then $\ufG_y\conn $ contains an
elliptic torus $\ufT\cong \UU(1,l/k)$.
For any character $\chi$ of $\finiteT$ of order greater than $2$
(i.e., $\chi$ with trivial stabilizer in
$W_k(\ufT ,\ufG_y\conn)$),
we have a
cuspidal Deligne-Lusztig representation of $\finiteG_y\conn $ of
degree $q-1$ whose character is
$-R_\ufT^{\ufG_y\conn }\chi$.
We note that
$R_\ufT^{\ufG_y\conn }\chi = R_\ufT^{\ufG_y\conn }\chi'$
if and only if $\chi'=\chi^{\pm 1}$.

\subsubsection{The group $\finiteG_z$}
\label{sub:finiteG_z}
The group $\tG_z$ is the intersection of the stabilizers
in $\GL(3,E)$ of two lattices:
$\LL'  := \OK_E \oplus          \OK_E \oplus \varpi_E \OK_E$,
and
$\LL'' := \OK_E \oplus \varpi_E \OK_E \oplus \varpi_E \OK_E$.
Explicitly, we have
$$
\tG_z = \GL(3,E) \cap
        \left[
        \begin{smallmatrix}
        \OK_E & \OK_E & \pp_E\inv \\
        \pp_E & \OK_E & \OK_E  \\
        \pp_E & \pp_E & \OK_E
        \end{smallmatrix}
        \right],
\:\text{and}\:
\tG_{z+} = 
         1 +
        \left[
        \begin{smallmatrix}
        \pp_E & \OK_E & \OK_E  \\
        \pp_E & \pp_E & \OK_E \\
        \pp^2_E & \pp_E & \pp_E
        \end{smallmatrix}
        \right].
$$
In particular, $\widetilde{\finiteG}_z \cong \GL(2,k)\times \GL(1,k)$.
Note that $\varepsilon$ acts on the first factor via
$g\mapsto g/\det(g)$, and on the second factor
via $g\mapsto g\inv$.
Thus,
$\finiteG_z \cong \SL(2,k) \times \{\pm 1\}$,
and
$\finiteG_z\conn \cong \SL(2,k)$.
As in the case of $\finiteG_y$, we have
$\finiteG_z=\finiteG_z\conn \times \finiteZ$.  Note that here the
factor $\finiteZ\cong\{\pm 1\}$ embeds as the group of scalar
matrices in $\finiteG_z$, not as the group
$\left(\begin{smallmatrix}
1 &   &  \\
  & \pm 1 &  \\
  &   & 1
\end{smallmatrix}\right)$.

As in the case of the vertex $y$, an irreducible representation of
$\finiteG_z$ can be expressed in the form $\tau\otimes\nu$, where $\tau$ is
an irreducible representation of $\finiteG_z\conn $ and $\nu$ is a
character of $\finiteZ$.
The representations of $\finiteG_z\conn $ are as follows.

If $\chi$ is a character of $\finiteM$, then the representation
$\ind_{\finiteB_z}^{\finiteG_z\conn }\chi$ is irreducible unless
$\chi$ has order dividing $2$.
In this case, the representation is a sum of two irreducible components.
If $\chi$ is trivial,
then these components are ${\bf 1}$ and
$\St_{\finiteG_z\conn }$, while if $\chi$ has order $2$, then the
components $\vartheta$ and $\vartheta'$ each have degree $(q+1)/2$.

Again, $\ufG_z\conn$ contains an
elliptic torus $\ufT \cong\UU(1,l/k)$. 
For any character $\chi$ of $\finiteT$ of order greater than $2$, we have a
cuspidal Deligne-Lusztig representation of $\finiteG_z\conn $ of
degree $q-1$ whose character is
$-R_\ufT^{\ufG_z\conn }\chi$.
As with $\finiteG_y$, $R_\ufT^{\ufG_z\conn }\chi = R_\ufT^{\ufG_z\conn }\chi'$
if and only if $\chi'=\chi^{\pm 1}$.
The remaining (cuspidal) representations of $\finiteG_z\conn$
are the two degree-$(q-1)/2$
components of $-R_\ufT^{\ufG_z\conn }\chi$ when $\chi$ has order $2$.

\subsubsection{The group $\finiteG_\F$}
The group $\tG_\F$ is the standard, upper-triangular Iwahori subgroup
of $\GL(3,E)$,
so $\widetilde{\finiteG}_\F$ is the reductive quotient
of the upper-triangular Borel subgroup of $\widetilde{\finiteG}_y$.
Therefore,
$$
\finiteG_\F =
\set{\left(
        \begin{smallmatrix}
        a&&\\&\pm1&\\&&a
        \end{smallmatrix}
\right)}{a\in k\mult}
\quad\text{and}\quad
\finiteG_\F\conn =
\set{\left(
        \begin{smallmatrix}
        a&&\\&1&\\&&a
        \end{smallmatrix}
\right)}{a\in k\mult}.
$$
We may thus write $\finiteG_\F = \finiteG_\F\conn \times \finiteZ$.
The representation theory of $\finiteG_\F$ is elementary.

\subsection{Maximal tori}
\label{sec:cartan-G}
Up to stable conjugacy, there are four families of maximal tori of
$\uG$.
According to the classification in~\cite{rogawski:u3},
they are isomorphic to:
\begin{enumerate}[(\ref{sec:cartan-G}--1)]
\addtocounter{enumi}{-1}
\item $R_{E/F}(\GL(1))\times \UU(1,E/F)$,
\item $\UU(1,E/F)\times \UU(1,E/F)\times \UU(1,E/F)$,
\item $R_{K/F}(\UU(1,EK/K))\times \UU(1,E/F)$ for $K\neq E$
a quadratic extension of $F$,
\item $R_{K/F}(\UU(1,EK/K))$ for $K$ a cubic extension of $F$.
\end{enumerate}
Note that all maximal tori of $\uG$ are elliptic except for those of
type~\cartan{0}.

\subsubsection{A torus of type~\cartan{1}}
\label{sec:checker_torus}
Up to conjugacy, $\underline{H}$ contains two $F$-tori in the
family~\cartan{1} (see~\cite{rogawski:u3}*{\S3.6}).
Let $\uC$ be the maximal $F$-torus of $\uH$ whose
group of $F$-points $C$ consists of those matrices $\gamma$ of the form
$$
  \gamma = \left(\begin{array}{ccc}
    \frac{\gamma_1+\gamma_3}{2} & 0 & \frac{\gamma_1-\gamma_3}{2}\\
    0 & \gamma_2 & 0\\
    \frac{\gamma_1-\gamma_3}{2} & 0 & \frac{\gamma_1+\gamma_3}{2}
  \end{array}\right)
$$
where $\gamma_i\in E^1$.  The group $\uC$ is of type \cartan{1},
and we identify
$C$ with
$E^1\times E^1\times E^1$ via the map $\gamma\mapsto
(\gamma_1 ,\gamma_2 ,\gamma_3 )$.  The Weyl group
$W(C,G)$ (resp.~$W(C,H)$) acts on $C$ by permuting the $\gamma_i$
(resp.~$\gamma_1$ and $\gamma_3$) transitively.  We note that $\uC$ is
an endoscopic group of $\uH$. 

\subsubsection{Ramified quadratic tori}
\label{sec:quadratic_torus}
Up to conjugacy, $\uH$ contains one $F$-torus of
type~\cartan{2} with $K=E'$ (see~\cite{rogawski:u3}*{\S3.6}).
Let $\underline{S}$ be the maximal $F$-torus of $\underline{H}$ such
that the group of $F$-points $S$ consists of those matrices of the form
$$
\begin{pmatrix}
\frac{1}{2}(\alpha+\iota'(\alpha)) & 0 &
        \frac{1}{2\varpi_{E'}}(\alpha-\iota'(\alpha))\\
0 & \gamma & 0\\
\frac{\varpi_{E'}}{2}(\alpha-\iota'(\alpha)) & 0 &
        \frac{1}{2}(\alpha+\iota'(\alpha))
\end{pmatrix} ,
$$
where $\alpha\in\ker (N_{EE'/E'})$, $\gamma\in\ker(N_{E/F})$ and
$\iota'$ is the generator of the group $\Gal (EE'/E)$.
Then $\underline{S}$ is a torus of the above type and $S\subset\stab_G(z)$.
Moreover, the group $\finiteS = \finiteS\conn\times\{\pm
1\}\subset\finiteG_z$  is isomorphic to
$\UU(1,l/k)\times\{\pm 1\}$, where $l$ is a quadratic
extension of $k$.

\begin{prop}
\label{prop:quadratic_tori}
The group $\uG$ contains exactly two $F$-conjugacy
classes of $F$-tori of type~\cartan{2} with $K=E'$.  There exists a
torus $\underline{S}'\subset\uG$ of this type such that $S'$ fixes the vertex
$y$.  Moreover, $\underline S'$ is stably conjugate but not $F$-conjugate
to $\underline{S}$.
\end{prop}

The first statement follows from~\cite{rogawski:u3}*{\S3.6}.
To establish the existence of $\underline{S}'$, one can
use the construction of
Morris~\cite{morris:symplectic}*{\S1}.
We omit the details.

We will need the following lemma concerning the values of
characters of $\finiteS\conn$ at regular
elements.  We leave the
elementary proof to the reader.
\begin{lem}
\label{lem:finite_torus}
Let $\chi$ and $\psi$ be nontrivial characters of $\finiteS\conn $,
a cyclic group of order $q+1$, and
suppose that the order of $\chi$ is greater than $2$.  Let $X$ denote
the set of all
$s\in\finiteS\conn $ with $s^2\neq 1$.
Let $m$ be any map from $X\longrightarrow$ to $\{\pm 1\}$.
\begin{enumerate}[(i)]
\item
\label{item:finite-torus-chichiinv}
The function $\chi+\chi\inv$ cannot vanish at every element of
  $X$ if $q>3$.
\item
\label{item:finite-torus-chichiinv-psipsiinv}
The functions $\chi+\chi\inv$ and $-\psi-\psi\inv$
  cannot agree on all of $X$ if $q\neq 7$.
\item
\label{item:finite-torus-chi-psi}
If $\chi+\chi\inv = \psi+\psi\inv$
  on $X$, then $\chi = \psi^{\pm 1}$.
\item
\label{item:finite-torus-chi-psi-mu}
If $\chi+\chi\inv = m\cdot (\psi+\psi\inv)$ on $X$, then $\chi = \psi^{\pm
  1}\mu$, where $\mu$ has order dividing 2.  If, in addition, the range of $m$
  is $\{\pm 1\}$ and $q>7$, then $\mu$ is nontrivial.
\item
\label{item:finite-torus-chichiinvpsi}
The functions $\chi+\chi\inv$ and $m\psi$ cannot
  agree on all of $X$.
\end{enumerate}
\end{lem}

\subsubsection{Ramified cubic tori}
\label{sec:cubic_torus}
Let $K$ be the extension of $F$ obtained by adjoining a cube root
$\varpi_K$ of $\varpi_F$ in $\overline F$, and let
$\underline{T}$ be an $F$-torus of type~\cartan{3} associated to $K$.
Then by~\cite{rogawski:u3}*{\S3.6}, $\underline{T}$ is the unique such
torus up to conjugacy.
Let $K'=EK = E(\varpi_{K'})$, where $\varpi_{K'}$ is the uniformizer
$\varpi_E/\varpi_K$ of $K'$.
Then $T\cong\ker N_{K'/K}$ and $\tT\cong(K')\mult$.

We will realize
the latter isomorphism (and the hence the torus $\widetilde{\underline{T}}$)
explicitly by
specifying that it take $a+b\varpi_{K'}+c\varpi_{K'}^2$ ($a,b,c\in E$
not all $0$) to
$$
\begin{pmatrix}
a & b & c\\
c\varpi_E & a & b\\
b\varpi_E & c\varpi_E & a
\end{pmatrix}
\in \GL(3,E) .
$$
We note that under this isomorphism, the
action of $\varepsilon$ on $\widetilde{T}$ corresponds to the action
$\alpha\mapsto\iota(\alpha)\inv$ on $(K')\mult$, where we view $\iota$
as the generator of $\Gal(K'/K)$.
It is easily checked that
$\widetilde{T}\subset\stab_{\tG}(\widetilde{\F})$.
Hence $T\subset\stab_{G}(\F)$ so $T_0$ is a subgroup of both $G_y$ and
$G_z$.

\begin{lem}
\label{lem:cubic_torus}
For $x=y$ or $z$, the image of $T_0$ in $\finiteG_x$ is the unipotent
radical of $\finiteB_x$.
\end{lem}

The proof is a straightforward calculation.

\subsection{Depth-zero representations of $G$}
\label{sec:depth-zero-G}

\subsubsection*{Principal series of $G$}

For $\lambda\in\Hom (M,\C\mult)$, there exist
unique characters $\lambda_1\in\Hom (E\mult,\C\mult)$ and
$\lambda_2\in\Hom (E^1,\C\mult)$ such that 
\begin{equation}
\label{eqn:lambda}
\lambda\left(\left(
\begin{smallmatrix}
          \alpha & 0 & 0\phinv \\
          0 & \beta & 0\phinv \\
          0 & 0 & \iota(\alpha)\inv
  \end{smallmatrix}
\right)\right)  =
\lambda_1(\alpha)\lambda_2(\alpha\iota(\alpha)\inv\beta ),
\end{equation}
where $\alpha\in E\mult$ and $\beta\in E^1$.
By~\cite{keys:thesis}, $\ind_B^G\lambda$ is
irreducible except for in the following cases:
\begin{enumerate}[(\ref{sec:depth-zero-G}PS--1)]
\item $\lambda_1 = |\cdot|_E^{\pm 1}$
\item $\lambda_1|_{F\mult} = \omega_{E/F} \, |\cdot |_F^{\pm 1}$
\item $\lambda_1$ is nontrivial and $\lambda_1|_{F\mult}$ is trivial.
\end{enumerate}

We note that in cases~\reducible{1} and~\reducible{2},
changing the sign of the exponent has the effect of
changing $\lambda$ into one of its conjugates under the
action of the Weyl group.
Thus,
the choice of exponent can affect the representation $\ind_B^G\lambda$,
but not its set of irreducible constituents.

In case~\reducible{1},
$\ind_B^G\lambda$ has two constituents: the
one-dimensional representation $\psi=\lambda_2\circ\det$, and the
square-integrable representation $\St_G\cdot\psi$.

In case~\reducible{2},
$\ind_B^G\lambda$ also has two constituents: a square-integrable
representation $\pi^2(\lambda)$ and a non-tempered unitary
representation $\pi^n(\lambda)$.

In case~\reducible{3},
$\ind_B^G\lambda$ decomposes into a direct sum
$\pi_1(\lambda)\oplus\pi_2 (\lambda)$.

By~\cite{moy-prasad:jacquet},
$\ind_B^G\lambda$ has depth zero if and only if
$\lambda$ has depth zero.
Hence, in each of these cases, $\lambda_2$ is a character of
$E^1/(E^1\cap (1+\pp_E))\cong\{\pm 1\}$
so there are only two
possibilities for $\lambda_2$.
Thus in case~\reducible{1}, there are only two Weyl group orbits of
characters $\lambda$, hence four distinct constituents of principal series:
two characters of $G$ and two special representations.

In case~\reducible{2}, there are two possibilities for
$\lambda_1|_{F\mult}$.  Since $\lambda_1$ is trivial on $1+\pp_E$ and
$|E\mult/F\mult(1+\pp_E)| = 2$, there are four possibilities for
$\lambda_1$ on $E\mult$, hence four Weyl group orbits of characters
$\lambda$.  Therefore, there are four distinct
representations $\pi^2(\lambda)$ and four distinct $\pi^n(\lambda)$.

In case~\reducible{3}, $\lambda_1$ must equal $\omega_{K/E}$, where
$K/E$ is unramified quadratic.  Thus there are two possibilities for
$\lambda$, and these characters are not in the same Weyl group orbit.
Therefore, there are two distinct
representations $\pi_1(\lambda)$ and two distinct $\pi_2(\lambda)$.

\subsubsection*{Supercuspidal representations of $G$}
\label{sec:depth-zero-sc}
Since $G$ has no non-minimal proper parabolic subgroups, the remaining
depth-zero irreducible representations are supercuspidal.
Any such
representation has a unique expression of the form
$\ind_{\stab_G(x)}^G\sigma$,
where $x$ is either $y$ or $z$, and $\sigma$ is
the inflation to $\stab_G(x)$ of a cuspidal representation of
$\finiteG_x$.
We can write the latter
representation as $\bar\sigma\otimes\nu$ for some character
$\nu$ of $\finiteZ$ and some cuspidal representation $\bar\sigma$
of $\finiteG_x\conn $ (see~\S\ref{sub:finiteG_y} and~\S\ref{sub:finiteG_z}).
Based on the classification of such
cuspidal representations (see~\S\ref{sub:reductive-quotients}), we have the
following kinds of supercuspidal representations of depth zero:
\begin{enumerate}[(\ref{sec:depth-zero-G}SC--1)]
\item $\ind_{\stab_G(y)}^G \sigma$, where $\bar\sigma$ is a cuspidal
  representation of $\finiteG_y$.
\item $\ind_{\stab_G(z)}^G \sigma$, where $\bar\sigma$ is a cuspidal
  representation of $\finiteG_z$ of degree $q-1$.
\item $\ind_{\stab_G(z)}^G \sigma$, where $\bar\sigma$ is a cuspidal
  representation of $\finiteG_z$ of degree $(q-1)/2$.
\end{enumerate}

As discussed in \S\ref{sub:reductive-quotients},
the representation $\bar\sigma$ of $\finiteG_x\conn $
($x=y$ or $z$) is a component of a cuspidal representation arising via
Deligne-Lusztig induction from a regular character $\chi$ of an elliptic
torus in $\finiteG_x\conn $.
We note that this torus can be
taken to be $\finiteS\conn $ if $x=z$ and $(\finiteS')\conn $ if $x=y$.
In case~\depthzerosc{3}, $\chi$ is the character $\phi$ of
$\finiteS\conn $ of order $2$.

\subsection{Base Change and Endscopic Transfers}
\label{sub:transfer}
In \cite{rogawski:u3}, Rogawski establishes the base change transfers
of $L$-packets from $H$ to $\tH$ and from $G$ to $\tG$ associated to
the natural embeddings $\lsup L\uH\rightarrow\lsup L\utH$ and $\lsup
L\uG\rightarrow\lsup L\utG$, respectively.  In addition, a certain
$L$-homomorphism $\lsup L \uH\rightarrow\lsup L \uG$ (resp.~$\lsup L
\uC\rightarrow\lsup L \uH$) is associated to any character $\varsigma$
of $E\mult$ extending $\omega_{E/F}$ (see
\cite{rogawski:u3}*{\S4.8.1}), and a corresponding transfer of
$L$-packets from $H$ to $G$ (resp.~$C$ to $H$) is established.  Let
$\Xi_G$ (resp.~$\Xi_H$) denote this transfer of $L$-packets in the
particular case $\varsigma = \Omega$.  

Due to the particular form of these $L$-homomorphisms, the Langlands
philosophy suggests
Hypothesis \ref{hyp:depth},
which we are assuming throughout.

\section{Description of depth-zero $L$-packets and explicit base
  change for unitary groups in two variables}
\label{sec:u11}
In this section we describe the
$L$-packets for the quasi-split
group $H^0 = U(1,1)(F)$
(defined with respect to the Hermitian form whose matrix is
$\left(
\begin{smallmatrix}
0&1\\1&0
\end{smallmatrix}
\right)$
as well as their base change lifts to $\GL(2,E)$.

We do the same for the group $\uH^1(F)$, where
$\uH^1$ denotes the $F$-anisotropic inner form of $\uH^0$.
One can realize $\uH^1(F)$ explicitly as the compact
unitary group $U(2)$,
defined with respect to the Hermitian form whose matrix is
$
\left(\begin{smallmatrix}
1&0 \\
0&-\epsilon
\end{smallmatrix}\right)
$, where $\epsilon$ is a nonsquare unit in $F\mult$.
Note that
$\utH^0(F) \cong \utH^1(F) \cong \GL(2,E)$.

We note that this explicit description of base change
implies that the analogue of Theorem~\ref{thm:main} holds for unitary groups
in two variables,
although we omit the verifications as they are entirely analogous to
(but less complicated than) those for $U(2,1)$.

Recall that $\uH = \uH^0\times \uZ$.
For every subgroup $\underline{L}$ of $\uG$, let
$\underline{L}^0$ denote the subgroup of $\uH^0$ obtained by
projecting $\underline{L}\cap\uH$ onto the $\uH^0$ component
of $\uH$.  

\subsection{Depth-zero $L$-packets of $U(1,1)$}
\label{sec:L-packets_u11}
The $L$-packets of $H^0$ are the orbits of $\PGL(2,F)$ on the
set of equivalence classes of irreducible admissible representations
of $H^0$~\cite{rogawski:u3}*{\S11.1}.
We first describe the
principal series $L$-packets.

Let $\lambda\in\Hom (M^0,\C\mult) = \Hom (E\mult   ,\C\mult   )$.
According to~\cite{rogawski:u3}*{\S11.1}, the principal series
$\ind_{B^0}^{H^0}\lambda$ is irreducible except in the cases
\begin{enumerate}
\item
$\lambda |_{F\mult} = |\cdot|_F^{\pm 1}$
\item
$\lambda |_{F\mult} = \omega_{E/F} $.
\end{enumerate}

In the first case, $\ind_{B^0}^{H^0}\lambda$ has two constituents: the
one-dimensional representation $\psi = \mu\circ\det_{\uH^0}$, where
$\mu$ is the character of $E^1$ satisfying
$\mu\circ\N^{U(1)}|_{E\mult} = \lambda |\cdot |_E^{\mp 1/2}$;
and the representation $\St_{H^0}\cdot\psi$.
In the second case,
$\ind_{B^0}^{H^0}\lambda$ decomposes into a direct sum
$\pi_1(\lambda)\oplus\pi_2 (\lambda)$
of irreducible representations.
By~\cite{moy-prasad:jacquet},
$\ind_{B^0}^{H^0}\lambda$ has depth zero if and only if $\lambda$ has
depth zero.

The principal series $L$-packets of $G$ are as
follows~\cite{rogawski:u3}*{\S11.1}.
(Here $\lambda$ and $\psi$ denote one-dimensional
representations of $M^0$ and $H^0$, respectively.)
\begin{enumerate}
\item $\{ \ind_{B^0}^{H^0}\lambda\}$, where $\ind_{B^0}^{H^0}\lambda$
  is irreducible;
\item $\{\psi\}$; 
\item $\{\St_{H^0}\cdot\psi\}$;
\item $\{\pi_1 (\lambda),\pi_2 (\lambda)\}$, where
  $\ind_{B^0}^{H^0}\lambda$
is reducible of the second type described above.
\end{enumerate}

The remaining irreducible representations and $L$-packets of $H^0$ are
supercuspidal.
The depth-zero supercuspidals of $H^0$ have a
unique expression of the form $\ind_{H^0_z}^{H^0} \sigma$, where
$\sigma$ is the inflation to $H^0_z$ of an irreducible cuspidal
representation $\bar\sigma$ of $\finiteH^0_z\cong \SL(2,k)$.
Note that since $\finiteH_z^0\cong\finiteG_z\conn$ the results
of \S\ref{sub:reductive-quotients} on $\finiteG_z\conn $ apply to
$\finiteH^0$, as well.
Therefore, we have the following kinds of supercuspidal
representations of $H^0$.
\begin{enumerate}[(\ref{sec:L-packets_u11}SC--1)]
\item $\ind_{G_z}^G \sigma$, where $\bar\sigma$ is a cuspidal
  representation of $\finiteH^0_z$ of dimension $q-1$.
\item $\ind_{G_z}^G \sigma$, where $\bar\sigma$ is a cuspidal
  representation of $\finiteH^0_z$ of dimension $(q-1)/2$.
\end{enumerate}

It is easily checked that $\PGL (2,F)$ acts trivially on the
equivalence class of any supercuspidal of the form~\Hdepthzerosc{1},
hence these representations are elements of $L$-packets of size $1$.
On the other hand, if $\epsilon$ is a nonsquare unit of $F\mult$, then
$\smattwo\epsilon001$
exchanges the two classes consisting of representations of the
form~\Hdepthzerosc{2}.
Thus $H^0$ has one depth-zero supercuspidal $L$-packet of size $2$.

\subsection{Base change lifts for $U(1,1)$}
\label{sec:bc-u11}
By~\cite{rogawski:u3}*{\S11.4},
the base change lifts of principal series
$L$-packets of $H^0$ are as follows.
Let $\lambda\in\Hom
(M^0,\C\mult)$, and let $\tilde\lambda$ denote the character of
$\tM^0$ lifted from $\lambda$ (see~\S\ref{sub:shintani}).
\begin{enumerate}[(i)]
\item 
If $\ind_{B^0}^{H^0}\lambda$ is irreducible and
  $\ind_{\widetilde B^0}^{\widetilde H^0}\tilde\lambda$ is irreducible,
  then the base
  change lift of the $L$-packet $\{ \ind_{B^0}^{H^0} \lambda\}$ is
  $\ind_{\widetilde{B}^0}^{\widetilde{H}^0}\tilde\lambda$.
\item
If $\ind_{B^0}^{H^0}\lambda$ is irreducible but
  $\ind_{\widetilde B^0}^{\widetilde H^0}\tilde\lambda$ is reducible,
  it must be the case that $\lambda|_{F\mult} = |\phantom{x}|_F^{\pm
    1}\omega_{E/F}$, and
  the base
  change lift of the $L$-packet $\{ \ind_{B^0}^{H^0} \lambda\}$ is
  $\left(\lambda |\phantom{x}|_E^{\mp 1/2}\right)\circ \det_{\utH^0}$.
\item 
If $\lambda|_{F\mult} = |\cdot |_F^{\pm 1}$, let $\psi$ be the
  one-dimensional representation $\mu\circ\det_{\uH^0}$,
  where $\mu\circ\N^{U(1)}|_{E\mult} = \lambda |\cdot |_E^{\mp 1/2}$.
  Then the lift of
  the $L$-packet consisting of the constituent $\psi$ (resp., the
  constituent $\St_{H^0}\cdot\psi$) of
  $\ind_{B^0}^{H^0}\lambda$ is the one-dimensional constituent
  $\tilde{\psi} = (\lambda |\cdot |_E^{\mp
    1/2})\circ\det_{\utH^0}$ (resp., the
  constituent $\St_{\widetilde{H}^0}\cdot\tilde{\psi}$) of
  $\ind_{\widetilde{B}^0}^{\widetilde{H}^0}\tilde\lambda$.
\item If $\lambda|_{F\mult} = \omega_{E/F}$, then the lift of the
  $L$-packet $\{\pi_1(\lambda),\pi_2(\lambda)\}$ is
  $\ind_{\widetilde{B}^0}^{\widetilde{H}^0}\tilde\lambda$.
\end{enumerate}

Now let $\Pi$ be the supercuspidal $L$-packet comprising the two
representations of the form~\Hdepthzerosc{2}.

Using Hypothesis \ref{hyp:depth} and a process of elimination,
it can be checked that $\Pi$ is the transfer via $\Xi_H$
from the endoscopic group
$\uC^0$ of the character $\varphi = {\bf 1}\otimes\phi$ of $C^0\cong E^1\times
E^1$, where
$\phi$ is the nontrivial depth-zero character of $E^1$ of order $2$
(see~\cite{rogawski:u3}*{\S11.1}).
The torus $\utM$ is $\tH^0$-conjugate to the torus $\utC$, say $\utC =
\lsup{h}\tM$, where $h\in \tH^0$.  It
follows from~\cite{rogawski:u3}*{Prop.~11.4.1(a)}
that the base change
lift of $\Pi$ is the representation
$$
\bigl(\ind_{\lsup{h}\widetilde{B}^0}^{\tH^0}
\tilde\varphi\bigr)\cdot\bigl(\Omega\inv\circ\det_{\utH^0}\bigr),
$$
where $\tilde\varphi$ is the character of $\tC$ lifted from $\varphi$.
Note that there are two possible choices for $\Omega$,
but up to equivalence, this lift does not depend on the choice.

We now discuss the base change lifts of the singleton depth-zero
supercuspidal $L$-packets.
Let $\rho$ be a depth-zero supercuspidal
representation of $H^0$ arising from a cuspidal representation
$\bar\sigma$ of $\finiteH_z^0$ of degree $q-1$.  Then (as in the case
of $\finiteG_z\conn$),
$\theta_{\bar\sigma}=-R_{\ufS^0}^{\ufH_z^0}\chi$ for some regular
character $\chi$ of $\finiteS^0$.
Since the character $\tilde\chi$ of $\utfS^0$ lifted from $\chi$ via
$N_\varepsilon^{\utfS^0}$ is in general position, there is a cuspidal
representation $\tilde{\bar\sigma}$ of $\widetilde{\finiteH}_z^0$ such that
$\theta_{\tilde{\bar\sigma}}
= -R_{\widetilde{\ufS}^0}^{\widetilde{\ufH}_z^0}\tilde\chi$.
We note that $\tilde{\bar\sigma}$ depends only on $\bar\sigma$ and not
on the choice of $\chi$.  Let $\tilde\sigma$ be the inflation to
$\widetilde{H}_z^0$ of $\tilde{\bar\sigma}$.
We can extend $\tilde\sigma$ uniquely to a
representation (which we will also denote by $\tilde\sigma$) of
$\widetilde{H}_z^0\widetilde{Z}^0$ in such a way that
$\omega_{\tilde\sigma}(\varpi_E) = \chi(-1)$.
(Here we are
identifying $\widetilde{Z}^0$ with $E\mult$ and $Z^0$ with $E^1$.)

\begin{prop}
\label{prop:bc_H_sc}
The base change lift of $\{\rho\}$ is
$\ind_{\widetilde{H}_z^0\widetilde{Z}^0}^{\widetilde{H}^0}\tilde\sigma$.
\end{prop}

\begin{proof}
  Let $\tilde\rho_*$ be the base change lift of $\{\rho\}$ and let
  $\tilde\rho =
  \ind_{\widetilde{H}_z^0\widetilde{Z}^0}^{\widetilde{H}^0}\tilde\sigma$.
  According to~\cite{rogawski:u3}*{\S11.4},
  $\tilde\rho_*$ is the Shintani lift of $\tilde\rho$ so we can choose
  $\tilde\rho_*(\varepsilon)$ so that
  $\theta_{\tilde\rho_*,\varepsilon} = \theta_\rho\circ\N^{\uH^0}$.
  Choose $\gamma$ in $S^0$ such
  that the image $\bar\gamma$ of $\gamma$ in $\finiteH^0_z$ is
  regular.  It follows from~\S\ref{sec:cartan-G} that $S^0\cong\ker
  N_{EE'/E'}$ and $\tS^0\cong (EE')\mult$.  Furthermore, these
  isomorphisms can be chosen so that the map $\N^{\uS^0}:\tS^0\rightarrow S^0$
  corresponds to the map $(EE')\mult\rightarrow\ker N_{EE'/E'}$ given
  by $x\mapsto x/\iota(x)$.  Since the latter is surjective even when
  restricted to $\OK_{EE'}\mult$, $\N^{\uS^0}$ must map $\tS^0_0$ onto
  $S^0$.  Thus, there exists $\delta\in\widetilde{S}_0^0\cap
  (\widetilde{H}^0)\ereg$ with $\N(\delta)=\gamma$.
  By~\cite{adler-lansky:bc-u3-unram}*{Lemma 2.1},
  $\gamma$ is
  contained in a unique parahoric subgroup of $H^0$, namely $H_z^0$.
  By~\cite{adler-lansky:bc-u3-unram}*{Prop.~7.1}, it follows that
  $\theta_\rho(\gamma) = \theta_\sigma(\gamma)$.
  Therefore, 
\begin{equation}
\label{eq:twisted_char1}
\theta_{\tilde\rho_*,\varepsilon}(\delta) =
\theta_{\rho}(\gamma) = \theta_\sigma (\gamma) = \theta_{\bar\sigma}
(\bar\gamma) = -\chi(\bar\gamma)-\chi\inv(\bar\gamma),
\end{equation}
where the last equality is the Deligne-Lusztig character
formula (see~\cite{srinivasan:finite-book}*{Thm.~6.8}).
  
  From~\cite{rogawski:u3}*{Prop.~11.4.1(c)}, $\tilde\rho_*$ is square
  integrable and is the base change lift of a unique square integrable
  representation.  Hence by (iii) above, $\tilde\rho_*$ cannot be a
  twist of the Steinberg representation, and so it must be
  supercuspidal.
  From Hypothesis \ref{hyp:depth}, $\tilde\rho_*$ has depth zero, so it must
  therefore be of the form
  $\ind_{\widetilde{H}^0_z\widetilde{Z}^0}^{\widetilde{H}^0}\tilde\sigma_*$,
  where $\tilde\sigma_*$ is an extension to $\widetilde{H}^0_z\widetilde{Z}^0$
  of the representation of $\widetilde{H}^0_z$ inflated from a cuspidal
  representation $\tilde{\bar\sigma}_*$ of $\widetilde{\finiteH}^0_z$.

  Since $\tilde\rho_*$ is $\varepsilon$-invariant, it follows
  from~\cite{moy-prasad:k-types}*{Thm.~5.2}
  that $\lsup\varepsilon \tilde{\bar\sigma}_*$ and $\tilde{\bar\sigma}_*$
  are associate, and this easily implies
  that $\tilde{\bar\sigma}_*$ is $\varepsilon$-invariant.
  Since $\tilde{\bar\sigma}_*$ is an irreducible cuspidal representation
  of $\widetilde{\finiteH}^0_z\cong\GL (2,k)$, it must be the case that
  $\tilde{\bar\sigma}_*$ arises via Deligne-Lusztig induction from a
  character $\varphi$ of $\widetilde{\finiteS}^0$ in general position
  with respect to the action of
$W(\widetilde\finiteS^0,\widetilde\finiteH_z^0)$.
Since
  $\tilde{\bar\sigma}_*$ is $\varepsilon$-invariant, we must have that
  $\lsup\varepsilon\varphi$ is in the
  $W(\widetilde\finiteS^0,\widetilde\finiteH_z^0)$-orbit
  of $\varphi$.  If $\lsup\varepsilon\varphi\neq\varphi$, it is easy to see
  that $\varphi$ has order dividing $2$, hence is not in general position.
Thus
  $^\varepsilon\varphi = \varphi$.
  It follows that $\varphi = \tilde\psi$ for some character
  $\psi$ of $\finiteS^0$ in general position with respect to
  the action of $W(\finiteS^0,\finiteH_z^0)$.
Let $\bar\sigma_*$ be the cuspidal
  representation of $\finiteH_z^0$ obtained via Deligne-Lusztig induction
  from $\psi$.

  Extend
  $\tilde\rho_*$ to a representation (also denoted
  $\tilde\rho_*$) of $\widetilde{H}^0\langle\varepsilon\rangle$
  such that for $\varepsilon$-regular $h\in\widetilde{H}^0$, we have
  $$
  \theta_{\tilde\rho_*}(h\varepsilon)
   = \theta_{\tilde\rho_*,\varepsilon}(h).
  $$
  As a representation of $\widetilde{H}^0\langle\varepsilon\rangle$,
  $$
  \tilde\rho_* =
  \ind_{\widetilde{H}_z^0\widetilde{Z}^0\langle\varepsilon\rangle}
  ^{\widetilde{H}^0\langle\varepsilon\rangle}
  \tilde\sigma_*,
  $$
  where $\tilde\sigma_*$ is extended compatibly from
  $\widetilde{H}_z^0\widetilde{Z}^0$ to
  $\widetilde{H}_z^0\widetilde{Z}^0\langle\varepsilon\rangle$.  This
  extension, in turn,
  determines an extension of $\tilde{\bar\sigma}_*$ to
  $\widetilde{\finiteH}_z^0\langle\varepsilon\rangle$.

Note that $\delta\in\tS^0_0\subset\widetilde{H}_z^0$.  Furthermore, an
argument in the proof
of~\cite{adler-lansky:bc-u3-unram}*{Prop.~5.5}
shows that $\delta\varepsilon$ lies
in a unique conjugate of
$\widetilde{H}_z^0\widetilde{Z}^0\langle\varepsilon\rangle$
(namely, $\widetilde{H}_z^0\widetilde{Z}^0\langle\varepsilon\rangle$ itself).
Therefore,
\cite{adler-lansky:bc-u3-unram}*{Prop.~7.1} implies
that
\begin{equation}
\label{eq:twisted_char2}
\theta_{\tilde\rho_*,\varepsilon}(\delta) = 
\theta_{\tilde\rho_*}(\delta\varepsilon) =
\theta_{\tilde\sigma_*}(\delta\varepsilon) =
\theta_{\tilde{\bar\sigma}_*}(\bar\delta\varepsilon),
\end{equation}
where $\bar\delta$ is the image of $\delta$ in $\widetilde{\finiteH}_z^0$.

To compute $\theta_{\tilde{\bar\sigma}_*}(\bar\delta\varepsilon)$,
consider the representation $^\varepsilon\tilde{\bar\sigma}_*$.
Since $\tilde\psi$ factors through $\N$, $\tilde\psi$ is trivial on
$\widetilde{\finiteZ}^0$, and so $\tilde{\bar\sigma}_*$ has trivial
central character.  As in the case of $\tfG_z$
(see~\S\ref{sub:finiteG_z}), $\varepsilon$ acts on $\tfH_z^0$ via
$x\rightarrow x/\det(x)$.  Therefore, for any
$h\in\widetilde{\finiteH}_z^0$, we have
that $^\varepsilon hh\inv = 1/\det h\in\widetilde{\finiteZ}^0$.
Thus $\tilde{\bar\sigma}_*(^\varepsilon hh\inv)$ is the identity so
that $^{\varepsilon}\tilde{\bar\sigma}_* = \tilde{\bar\sigma}_*$.
(Note that before, we only had that $^\varepsilon\tilde{\bar\sigma}_*$ and
$\tilde{\bar\sigma}_*$ were equivalent.)
In other words,
$\tilde{\bar\sigma}_*(\varepsilon) = \pm \text{id}$.
This together with \eqref{eq:twisted_char2} and the
Deligne-Lusztig character formula imply that
$\theta_{\tilde\rho_*,\varepsilon}(\delta) =
\theta_{\tilde{\bar\sigma}_*}(\bar\delta\varepsilon)$ is equal to
\begin{equation}
\label{eq:twisted_char3}
\pm\theta_{\tilde{\bar\sigma}_*}(\bar\delta) =
\pm(\tilde\psi(\bar\delta)+\,^w\tilde\psi(\bar\delta)) =
\pm(\psi(\bar\gamma)+\psi\inv(\bar\gamma)),
\end{equation}
where $w$ is the nontrivial element of
$W(\widetilde\finiteS^0,\widetilde\finiteH_z^0)$.
It follows that from \eqref{eq:twisted_char1} and
\eqref{eq:twisted_char3} that
$\chi(\bar\gamma)+\chi\inv(\bar\gamma) 
= \pm (\psi(\bar\gamma)+\psi\inv(\bar\gamma))$.
Letting $\gamma$ vary over all elements of $S^0$
with regular image in $\finiteH_z^0$, parts
(\ref{item:finite-torus-chichiinv-psipsiinv}) and
(\ref{item:finite-torus-chi-psi}) of Lemma~\ref{lem:finite_torus}
then imply that $\psi = \chi^{\pm 1}$ if $q\neq 7$
(and that the sign in~\eqref{eq:twisted_char3}
must be $+$).  Thus $\tilde{\bar\sigma}_*$ is equivalent to
$\tilde{\bar\sigma}$.

Moreover, since 
$\omega_{\tilde\rho_*} = \omega_\rho\circ\N^{\uZ^0}$,
$$
\omega_{\tilde\sigma_*}(\varpi_E)
= \omega_{\tilde\rho_*}(\varpi_E)
= \omega_\rho(-1)
= \omega_\sigma(-1)
= \chi(-1)
= \omega_{\tilde\sigma}(\varpi_E).
$$
It follows that the extensions of $\tilde\sigma$ and $\tilde\sigma_*$ to
$\widetilde{H}_z^0\widetilde{Z}^0$ are equivalent.
The theorem follows.
\end{proof}

\subsection{Depth-zero $L$-packets for $U(2)$}
The group $H^1=U(2)(F)$ is compact, so its building has only one point: $y$.
The corresponding parahoric subgroup has index two in $H^1$.
The quotient $H^1/{H^1_y}^+$ is isomorphic to the orthogonal group
$\OO(2)(k)$ defined with respect to the form
$\left(
\begin{smallmatrix}
1&0 \\
0&-\bar\epsilon
\end{smallmatrix}\right)
$
Here $\overline\epsilon$ is the image in $k$ of the nonsquare unit
$\epsilon\in F\mult$.
We may identify $\SO(2)(k)$ with the group of norm-one elements
of the quadratic extension of $k$.
Fix an element $h$ of $\OO(2)(k)$ that does not belong to $\SO(2)(k)$.
Then conjugation by $h$ acts on $\SO(2)(k)$ by inversion.
This allows us to classify the representations of $\OO(2)(k)$,
and thus the depth-zero representations of $H^1$,
and to gather these latter representations into $L$-packets.
Our $L$-packets are determined by characters of $\SO(2)(k)$.
Let $\chi$ denote such a character.
\begin{enumerate}
\item
Suppose that $\chi\neq \chi\inv$.
Then $\infl\ind_{\SO(2)(k)}^{\OO(2)(k)}\chi$ is irreducible, and forms
a singleton $L$-packet $\Pi^1_\chi$.
\item
Suppose that $\chi$ is trivial.
Then $\chi$ extends to $\OO(2)(k)$ in two ways,
yielding a trivial character $\chi_+$ and a quadratic
character $\chi_-$.
Inflating these extensions, we obtain two singleton
$L$-packets for $H^1$, which we will denote by
$\Pi^1_+$ and $\Pi^1_-$, respectively.
\item
Suppose that $\chi$ is
the nontrivial quadratic character of $\SO(2)(k)$.
Then again $\chi$ extends to $\OO(2)(k)$ in two ways.
Inflating each character to $H^1$, we obtain an $L$-packet
$\Pi^1_\chi$
of size two.
\end{enumerate}
It is easy to see that for each $L$-packet above,
the sum of the characters in an $L$-packet is stable.
Moreover, our $L$-packets are minimal with respect to
this property.

\subsection{Base change lifts for $U(2)$ via a Jacquet-Langlands-like
correspondence}
\label{sec:bc-u2}

Since $\underline{H}^1$ is an inner form of $\underline{H}^0$,
we can obtain a base change lift if we can associate
each depth-zero $L$-packet for $H^1$ to one for $H^0$.
This association will be similar to the Jacquet-Langlands
correspondence (see \cites{badulescu:thesis,dkv:amt,rogawski:amt}).
That is, given a depth-zero $L$-packet $\Pi^1$ for $H^1$,
we want to find a depth-zero $L$-packet $\Pi^0$ for $H^0$ such that
\begin{equation}
\label{eqn:jl-u2}
\sum_{\pi\in\Pi^1} \theta_\pi(g_1)
=
-
\sum_{\pi\in\Pi^0} \theta_\pi(g_0)
\end{equation}
for all regular $g_1\in H^1$
and $g_0\in H^0$ whose stable conjugacy classes are
associated in a natural way.

Define a map $\JL$ from the depth-zero $L$-packets of $H^1$
to those of $H^0$ as follows.
We let $\JL(\Pi^1_+)$ be the Steinberg representation
of $H^0$, and we let $\JL(\Pi^1_-)$ be the twist of
this representation by the depth-zero quadratic character of $H^0$.
If $\chi\neq\chi\inv$, 
then $\JL(\Pi^1_\chi)$
is the supercuspidal representation of the form \Hdepthzerosc{1}
coming from the character $\chi$.
If $\chi$ is the nontrivial quadratic character, then
$\JL(\Pi^1_\chi)$
is the $L$-packet consisting of the two representations
of the form \Hdepthzerosc{2}.

Using Lemma~\ref{lem:finite_torus},
it is not difficult to see that if $q>3$ and $q\neq 7$,
then $\JL$ is the only correspondence
that satisfies
\eqref{eqn:jl-u2}
for all $g_1 \in H^1$ whose image in $\SO(2)(k)$ is regular.
Thus, if we assume that there is a Jacquet-Langlands-like correspondence
from the depth-zero $L$-packets of $H^0$ to those of $H^1$, then
it must be $\JL$.

According to Conjecture 1 of \cite{dprasad:theta-unitary},
if it exists, then
$\JL$ is related to the theta correspondence between $H^0$ and $H^1$
in the following way.  Given an $L$-packet $\Pi$ for $H^1$, some
representation $\pi\in \Pi$ shows up in the theta correspondence,
and $\Theta(\pi) = \JL(\Pi)$.

\section{Description of depth-zero $L$-packets
for ramified unitary groups in
  three variables}
\label{sec:u21}

\subsection{$L$-Packets consisting of principal series constituents}
\label{sub:ps}
The following proposition is due to
Rogawski~\cite{rogawski:u3}*{\S12.2}.

\begin{prop}
\label{prop:ps_L-packets}
The $L$-packets of $G$ that consist entirely of principal series
constituents all have one of the following forms
(where $\lambda$ and $\psi$ denote one-dimensional representations
of $M$ and $G$, respectively):
\begin{enumerate}[(\ref{sub:ps}--1)]
\item  $\{ \ind_P^G\lambda\}$,
  where $\ind_P^G\lambda$ is irreducible;
\item $\{\psi\}$; 
\item
$\{\St_G\cdot\psi\}$;
\item $\{\pi_1 (\lambda),\pi_2 (\lambda)\}$, where $\ind_P^G\lambda$ is
  reducible of type \textup{\reducible{3}}.
\item $\{\pi^n (\lambda)\}$, where $\ind_P^G\lambda$ is
  reducible of type \textup{\reducible{2}}.
The representation $\pi^n (\lambda)$
is contained in the $A$-packet
$\Pi (\lambda) = \{ \pi^n (\lambda),\pi^s (\lambda)\}$,
where $\pi^s (\lambda)$
is a supercuspidal
representation that sits inside an $L$-packet with the 
square-integrable principal series constituent $\pi^2(\lambda)$.
\end{enumerate}
\end{prop}

In the depth-zero setting,
the representation $\pi^s (\lambda)$ will be
explicitly described in~\S\ref{sub:non-sc}.

\subsection{Supercuspidal $L$-packets}
\label{sub:sc}
In order to compute the depth-zero supercuspidal $L$-packets of $G$,
we need two preliminary results.  The first gives the values of the
characters of depth-zero supercuspidal representations on ``very
regular'' elements of certain elliptic tori.  The second determines
the sizes of such $L$-packets and asserts that they are all transfers
via $\Xi_G$ of $L$-packets of $H$.

Let $\underline{S}$ and $\underline{S}'$ be the ramified quadratic
tori defined in \S\ref{sec:quadratic_torus}.
Let $\gamma\in S$ be a regular element of $G$ whose image $\bar\gamma$
in $\finiteS\cong\finiteS\conn \times\finiteZ$
lies in $\finiteS\conn $ and is regular in $\finiteG_z\conn $.
Then~\cite{adler-lansky:bc-u3-unram}*{Lemma 2.1}
implies
that $\gamma$ lies in a unique parahoric subgroup of $G$,
namely $G_z$.
By Proposition~\ref{prop:quadratic_tori}, there is an element $g\in\uG(E)$
such that $gSg\inv = S'$.  Set $\gamma' = g\gamma g\inv$.
Then similarly,
we have that $\gamma'$ lies in a unique maximal compact subgroup of
$G$, namely $G_y$.

The next lemma concerns character values of supercuspidal
representations of depth zero.
Recall that a regular character $\chi$
of $\finiteS\conn $ or $(\finiteS')\conn $ together with a character
$\nu$ of $\finiteZ$ gives rise to a supercuspidal representation of $G$
as explained in~\S\ref{sec:depth-zero-sc}.
\begin{lem}
\label{lem:char_values}
  Let $\pi$ be a supercuspidal representation of $G$ of depth zero.
  Then, in the notation of \S\ref{sec:depth-zero-sc},
\begin{align*}
\theta_\pi(\gamma) &= 
\begin{cases}
0 & \text{if $\pi$ is of type~\depthzerosc{1}},\\
-\left[\chi(\bar\gamma) + \chi\inv(\bar\gamma)\right]
& \text{if $\pi$ is of type~\depthzerosc{2}},\\
-\phi(\bar\gamma)
& \text{if $\pi$ is of type~\depthzerosc{3}}, 
\end{cases}
\\
\theta_\pi(\gamma') &=
\begin{cases}
-\left[\chi(\bar\gamma) + \chi\inv(\bar\gamma)\right]
& \text{if $\pi$ is of type~\depthzerosc{1}},\\
0 & \text{if $\pi$ is of type~\depthzerosc{2}},\\
0 & \text{if $\pi$ is of type~\depthzerosc{3}}. 
\end{cases}
\end{align*}
where $\phi$ is the character of $\finiteS\conn $ of order $2$.
\end{lem}

\begin{proof}
  By the discussion
  in~\S\ref{sec:depth-zero-sc}, $\pi$ is compactly induced from a
  representation $\sigma$ of $G_v$, where $v=y$ or $z$.  Let
  $\gamma^*$ be either $\gamma$ or $\gamma'$ and, correspondingly, let
  $x$ be either $z$ or $y$.  Since $\gamma^*$ lies in $G_x$ and no
  other parahoric subgroup,
  \cite{adler-lansky:bc-u3-unram}*{Prop.~7.1} implies
  that $\theta_\pi(\gamma^*) = \theta_\sigma(\gamma^*)$ if $v=x$ and
  $\theta_\pi(\gamma^*)=0$ otherwise.  In the former case
  $\theta_\sigma(\gamma^*) = \theta_{\bar\sigma}(\bar\gamma^*)$, where
  $\bar\gamma^*$ is the image of $\gamma^*$ in $\finiteG_x\conn $.
  In cases~\depthzerosc{1} and~\depthzerosc{2}, the character of
  $\bar\sigma$ is (up to a sign) the Deligne-Lusztig character of
  $\finiteG_x^\circ$ corresponding to $\chi$, and the result follows
  from the the character 
  formula for such
  representations~\cite{srinivasan:finite-book}*{Thm.~6.8}.  In
  case~\depthzerosc{3}, $\bar\sigma$ is one of the two irreducible
  components of the Deligne-Lusztig representation of
  $\finiteG_z^\circ$ corresponding to $\phi$, and the result follows
  from this character formula together with the fact that these two
  components take the same value on semisimple elements of $\finiteG_z^\circ$.
\end{proof}

\begin{lem}
\label{lem:sc-packet-size}
Let $\rho$ be a depth-zero supercuspidal representation of $H$ such
that $\{\rho\}$ is an $L$-packet.  Then the transfer $\Pi$ of
$\{\rho\}$ via $\Xi_G$  is a depth-zero supercuspidal $L$-packet of $G$ of
size two.  Moreover, every depth-zero supercuspidal
$L$-packet $\Pi$ arises in this way.
\end{lem}

\begin{proof}
The first statement follows from~\cite{rogawski:u3}*{\S12.2} given our
assumption on $\Xi_G$ in Hypothesis \ref{hyp:depth}.

Now suppose that $\Pi$ is a depth-zero supercuspidal $L$-packet of
$G$.  An immediate consequence of Lemma~\ref{lem:char_values}
is that none of the depth-zero
supercuspidal characters are stable.
Since the sum of the characters of the elements of a tempered
$L$-packet is stable (see~\cite{rogawski:u3}*{\S12--13}), it follows
that $G$ has no
depth-zero supercuspidal $L$-packets of size $1$.
According
to~\cite{rogawski:u3}*{\S12.2, Lemmas~13.1.1 and 13.1.3}
and Hypothesis \ref{hyp:depth},
a non-singleton
depth-zero supercuspidal $L$-packet $\Pi$ must be the transfer via
$\Xi_G$ of a
depth-zero supercuspidal $L$-packet $R$ of the endoscopic group $H$.
From~\cite{rogawski:u3}*{\S12.1},
we have that
the $L$-packet $R$ 
has size $1$ or $2$.
But if $|R|$ were
$2$, then $R$ would be the transfer via $\Xi_H$ of a (depth-zero)
$W(C,H)$-regular
character $\psi$ of $C$.
Moreover, since $\Pi$ is supercuspidal $\psi$ would have to be
$W(C,G)$-regular.  But since $W(C,G)$ acts on $C = E^1\times E^1\times
E^1$ by permuting coordinates
transitively, and since since $\finiteC \cong \{\pm 1\}^3$, $C$ has no
$W(C,G)$-regular characters of depth-zero.  Hence $|R|$
must be $1$ and
$|\Pi|=2\,|R|=2$ by~\cite{rogawski:u3}*{Lemma 13.1.2}.
\end{proof}

Using the results of \S\ref{sec:L-packets_u11}
and~\cite{rogawski:u3}*{\S12.1}, we see that the
singleton supercuspidal depth-zero $L$-packets of $H$ are those of the form
$\{\rho\}$, where $\rho = \rho^0\otimes\psi$, $\rho^0$ is a
supercuspidal of $H^0$ of type~\Hdepthzerosc{1}, and $\psi$ is a
depth-zero character of $Z$.

We now determine the depth-zero supercuspidal $L$-packets of $G$.
In the following, we will view the restriction of $\Omega$
to $E^1$ as a character of $Z$; hence $\bar\Omega$ is a character
of $\finiteZ$.
Also, for any character $\nu$ of $\finiteZ$,
$\widehat\nu$ will denote the inflation of $\nu$ to $Z$.
Now fix a characer $\nu$ of $\finiteZ$, and let $\chi$ be a character
of $\finiteS^0$ of order greater than $2$.
Note that $\finiteS^\circ = \finiteS^0$, so we may view $\chi$ as a
character of $\finiteS^\circ$.  By identifying $\finiteS$ with
$\finiteS'$ via conjugation by the element $g\in \uG(E)$ from the
beginning of this section, we may also view
$\chi$ as a character of $(\finiteS')^\circ$.
(This character of $(\finiteS')^\circ$ depends on the choice of $g$,
but its Weyl group orbit is uniquely determined by $\chi$.  The
following representations constructed from $\chi$ depend only on this
orbit.)  We fix an identification of $\finiteS^0$ with $l^1$, where $l/k$ is a
quadratic extension.  This determines identifications
of $\finiteS^\circ$ and $(\finiteS')^\circ$ with $l^1$.

If $x=y$ (resp.~if $x=z$), let $\pi_x= \pi_x(\chi,\nu)$ be the supercuspidal
representation of $G$ of the form~\depthzerosc{1}
(resp.~\depthzerosc{2}) associated to $\chi$ and $\nu$.
That is,
$\pi_x =
\ind_{\stab_G(x)}^{G}
        \infl(\bar\sigma_x\otimes\nu)$,
where $\bar\sigma_x$ is the cuspidal Deligne-Lusztig representation of
$\finiteG_x^\circ$ obtained from $\chi$.
Let $\rho = \rho(\chi,\nu) = \ind_{\stab_H(z)}^H\tau$, where $\tau$ is
defined as follows.
Let $\bar\tau$ be the Deligne-Lusztig
cuspidal representation of $\finiteH_z\conn $ associated to $\chi\phi$
($\phi$ the character of $\finiteS^0$ of order $2$), and
define $\tau$ to be the inflation to $\stab_H(z)$ of the representation
   $\bar\tau\otimes\nu\bar\Omega$ of $\finiteH_z = \finiteH_z\conn
\times\finiteZ$.

\begin{prop}
\label{prop:sc_card2}
The singleton $\{\rho(\chi,\nu)\}$
is a depth-zero supercuspidal $L$-packet of $H$.
The image of $\{\rho(\chi,\nu)\}$ under $\Xi_G$ is the
depth-zero supercuspidal $L$-packet
$\{\pi_y(\chi,\nu),\pi_z(\chi,\nu)\}$ of $G$.
Moreover, every depth-zero supercuspidal $L$-packet of $G$ is of this form.
\end{prop}

\begin{proof}
By Lemma~\ref{lem:sc-packet-size} and the subsequent discussion,
$\{\rho\}$ is an
$L$-packet of $H$ whose transfer via $\Xi_G$ is a supercuspidal $L$-packet
$\Pi = \{\pi_1,\pi_2\}$ of size $2$, and all such $L$-packets arise in
this way.  Therefore, to prove the proposition, it suffices to show
that $\Pi = \{\pi_y,\pi_z\}$.

From Hypothesis \ref{hyp:depth}, $\Pi$ has depth zero.
As discussed in~\S\ref{sec:depth-zero-G}, $\pi_i$ ($i=1,2$) is
associated to a cuspidal representation $\bar\sigma_i$ of
$\finiteG_{x_i}\conn$ ($x_i=y$ or $z$) and a character $\nu_i$ of
$\finiteZ$.  To determine
$\nu_i$, note that the central character $\omega_{\pi_i}$ of $\pi_i$ is equal
to $\widehat\nu_i$, while $\omega_\rho = \widehat\nu\Omega$.
By~\cite{rogawski:u3}*{Prop.~4.9.1},
$\omega_{\pi_i}$ must satisfy $\omega_{\pi_i} = \Omega\omega_\rho
= \Omega\widehat\nu\Omega = \widehat\nu$.
Thus $\nu_i=\nu$. 

We now determine the $\bar\sigma_i$.
By~\cite{rogawski:u3}*{Lemma 12.7.2}, up to
re-ordering, 
$$
\theta_\rho^G = \theta_{\pi_1}-\theta_{\pi_2},
$$
where
$\theta_\rho^G$ is the distribution on $G$ coming from $\theta_\rho$
via the transfer $\Xi_G$.
Let $\gamma$ be an element of $S^0$ whose
image $\bar\gamma$ in $\finiteG_z$ is a regular element of
$\finiteS^0$.
Let $\gamma' = g\gamma g\inv$
be as in Lemma~\ref{lem:char_values},
and let $\delta$ be the element of 
$\D_\uG(\underline{S}/F)$ represented by the cocycle
$s\mapsto s(g)g\inv$, $s\in\Gal (E/F)$.
Let $\chi_0 = \chi\phi$.
A result analogous to Lemma \ref{lem:char_values} shows that
$\theta_\rho(\gamma) = - \chi_0\vphinv(\bar\gamma) - \chi_0\inv(\bar\gamma)$.
Moreover, due to the
particular form of $\gamma$ and the fact
(see~\cite{rogawski:u3}*{Prop.~3.7.1})
that
$W_F(\underline{S},\uH) = W_F(\underline{S},\uG)$, it follows
from~\cite{rogawski:u3}*{\S4.9, Lemma 12.5.1} that
$$
\theta_\rho^G(\gamma) = n(\gamma)\theta_\rho(\gamma),
\qquad
\theta_\rho^G(\gamma') = n(\gamma)\kappa(\delta)\theta_\rho(\gamma),
$$
where $\kappa $ is the element of the dual of
$\D_\uG(\underline{S}/F)$ corresponding to the endoscopic group $\uH$,
and 
$$
n(\gamma) = \Omega (-N_{EE'/E}(\gamma^{-1}-1)) =
\bar\Omega(-N_{l/k}(\bar\gamma^{-1}-1)) = \bar\Omega(-N_{l/k}(\bar\gamma-1)).
$$
(In this formula, we are identifying
$\gamma\in S^0$ (resp.~$\bar\gamma\in\finiteS^0$) with the
element of $\ker (N_{EE'/E})$ (resp.~$l^1$) to which it corresponds.)

Note that $n(\gamma)\in\{\pm 1\}$.  Moreover, as $\gamma$ ranges over
all elements of $S^0$ with regular image in $\finiteG_z$, $\bar\gamma$
ranges over all of $l^1\smallsetminus\{\pm 1\}$, and an elementary
argument shows that $n(\gamma) = \bar\Omega(-N_{l/k}(\bar\gamma-1))$
must assume both the values $1$ and $-1$ if $q>5$.
We also note that $\delta$ is not in the image of
$\D_{\uH}(\underline{S},F)$, and consequently, we have
$\kappa(\delta)=-1$
by~\cite{rogawski:u3}*{\S4.10, Lemma~3.13}.
Thus we have
\begin{equation}
\label{eq:endo_lift_reg}
\begin{aligned}
\theta_{\pi_1}(\gamma\phprime) - \theta_{\pi_2}(\gamma\phprime) 
= \theta_\rho^G(\gamma\phprime) &
= \phm n(\gamma)\theta_\rho(\gamma)
= - n(\gamma)\left(\chi_0(\bar\gamma) + \chi_0\inv(\bar\gamma)\right) \,,
\\
\theta_{\pi_1}(\gamma') - \theta_{\pi_2}(\gamma')
= \theta_\rho^G(\gamma') &
= -n(\gamma)\theta_\rho(\gamma)
= \phm n(\gamma)\left(\chi_0(\bar\gamma) + \chi_0\inv(\bar\gamma)\right) \, .
\end{aligned}
\end{equation}

Assume that both $\pi_i$ are of the form \depthzerosc{1}.
Then it follows from Lemma~\ref{lem:char_values}
that
$\theta_{\pi_i}(\gamma)=0$.
By Lemma~\ref{lem:finite_torus}(\ref{item:finite-torus-chichiinv}),
this contradicts
\eqref{eq:endo_lift_reg} if $q>3$.
A similar argument with $\gamma'$
replacing $\gamma$ shows that
at least one of the representations $\pi_i$
must be of the form \depthzerosc{1}.
Hence precisely one of the $\pi_i$, say $\pi_j$, is of this form, and
the the other, say $\pi_{j'}$, is of the form
\depthzerosc{2} or \depthzerosc{3}.  

We then have $x_j=y$, and as above,
$\bar\sigma_j$ is associated to
a character $\chi_j$ of
$\finiteS^0\cong(\finiteS')^\circ$ of order greater
than $2$.
By Lemma~\ref{lem:char_values},
$\theta_{\pi_{j'}}$ vanishes at $\gamma'$ and by
\eqref{eq:endo_lift_reg}
we have
$$
\pm\left(\chi_j(\bar\gamma)-\chi_j\inv(\bar\gamma)\right)
= \theta_{\pi_1}(\gamma') - \theta_{\pi_2}(\gamma')
= n(\gamma)\left(\chi_0(\bar\gamma) + \chi_0\inv(\bar\gamma)\right).
$$
By Lemma~\ref{lem:finite_torus}(\ref{item:finite-torus-chi-psi-mu}),
this implies that $\chi_j = \chi_0^{\pm 1}\phi = \chi^{\pm1}$ if $q>7$.  Thus
$\bar\sigma_j\cong\bar\sigma_y$.
Since $\nu_j = \nu$, we have $\pi_j\cong\pi_y$.

Since $x_j=y$, it must be the case that $x_{j'} = z$.  
An argument similar to that in the preceding paragraph using
both parts (\ref{item:finite-torus-chi-psi-mu}) and
(\ref{item:finite-torus-chichiinvpsi}) of 
Lemma~\ref{lem:finite_torus}
shows that $\pi_{j'}$ must be of the form~\depthzerosc{2}, and that 
if $\bar\sigma_{j'}$ is associated to the character $\chi_{j'}$ of
$\finiteS^0 = \finiteS\conn$, then
$\chi_{j'}=\chi^{\pm1}$.  It follows that
$\bar\sigma_{j'}\cong\bar\sigma_z$, and hence that $\pi_{j'}\cong\pi_z$.
\end{proof}

\subsection{Non-supercuspidal $L$-packets containing supercuspidals}
\label{sub:non-sc}
It remains to determine the depth-zero
$L$-packets of $G$ of the form $\{\pi^s(\lambda),\pi^2(\lambda)\}$,
i.e., the non-supercuspidal depth-zero $L$-packets containing a
supercuspidal representation.  In this case, $\lambda$
is a depth-zero character of $M$ such that the restriction of
the character $\lambda_1$ of $E\mult$ to $F\mult$ is
$\omega_{E/F}|\cdot|^{\pm 1}_F$ (see~\reducible{2} and
\pslpacket{5}).  In other words, $\lambda_1 =
\eta|\cdot|^{\pm 1/2}_E$ for some character $\eta$ of $E\mult$
satisfying $\eta|_{F\mult} = \omega_{E/F}$.  Recall from
Proposition~\ref{prop:ps_L-packets} that the discrete series
representation $\pi^2(\lambda)$ is a constituent of the principal series
representation $\pi = \ind_B^G \lambda$,
the only other constituent being a non-tempered
representation, denoted $\pi^n(\lambda)$, which is paired with
$\pi^s(\lambda)$ in an $A$-packet.

The representation $\pi^s(\lambda)$ must be of the
form~\depthzerosc{3}, as those of the form~\depthzerosc{1}
and~\depthzerosc{2} are precisely the representations that lie in
depth-zero supercuspidal $L$-packets by Proposition~\ref{prop:sc_card2}.
There are four equivalence classes of representations of the
form~\depthzerosc{3} since there are two equivalence classes of
cuspidal representations of $\finiteG_z\conn $ of degree $(q-1)/2$,
and two characters of $\finiteZ\cong\{\pm 1\}$.  This agrees with our
determination in~\S\ref{sec:depth-zero-G} that there are four different
representations of the form $\pi^2(\lambda)$ (resp.~$\pi^n(\lambda)$).

In the following proposition, we determine $\pi^s(\lambda)$ in
terms of the representation $\pi^n(\lambda)$.

\begin{prop}
\label{prop:A-packet}
  Let $\lambda $ be a depth-zero character of $M$ such that
  $\lambda_1|_{F\mult} = \omega_{E/F} \, |\cdot |_F^{\pm 1}$.
  \begin{enumerate}[(i)]
\item
As a representation of $\finiteG_z\conn $,
$\pi^n(\lambda)^{G_{z+}}$ is equivalent to an irreducible
    component $\vartheta$ of the reducible principal series
    representation $\ind_{\finiteB_z}^{\finiteG_z}\chi$, where $\chi$
    is the character of $\finiteM\conn$ of order $2$.
\item
The representation $\pi^s(\lambda)$ is supercuspidal of the
    form~\depthzerosc{3}.
In particular,
$\pi^s(\lambda) \cong \ind_{\stab_G(z)}^G \infl(\bar\sigma \otimes \nu)$,
where $\nu$ is the character $\bar\lambda_1|_{\{\pm 1\}}\bar\lambda_2$
of $\finiteZ$, and
$\bar\sigma$ is the unique irreducible cuspidal representation of
    $\finiteG_z\conn $ whose character agrees 
with $-\theta_{\vartheta}$ on the set of nontrivial unipotent
    elements in $\finiteG_z\conn $.
  \end{enumerate}
\end{prop}
Note that, as one would expect, $\pi^s(\lambda)$ is independent of the power
of $|\cdot|_E$ occurring in $\lambda_1$.
\begin{proof}
As explained in the preceding discussion, $\pi^s = \pi^s(\lambda)$
must be supercuspidal of the form~\depthzerosc{3}.
That is,
$\pi^s(\lambda) = \ind_{\stab_G(z)}^G \infl(\bar\sigma\otimes\nu)$ for
some cuspidal representation $\bar\sigma$ of $\finiteG\conn_z$ of
degree $(q-1)/2$ and some character $\nu$ of $\finiteZ$.
Our task is to
determine $\bar\sigma$ and $\nu$ in terms of $\pi^n = \pi^n(\lambda)$.
 
From~\cite{rogawski:u3}*{Thm.~13.1.1, Prop.~13.1.2}, the distribution
characters of $\pi^s$ and $\pi^n$ satisfy
\begin{equation}
\label{eq:char_identity}
\theta_{\pi^s} = \theta_\xi^G - \theta_{\pi^n},
\end{equation}
where $\xi$ is a certain character of $H$ associated to $\lambda$ and
$\theta_\xi^G$ is the transfer of $\xi$ via $\Xi_G$.
The same equation holds for the functions
on $G\reg$ that represent these distributions.
Since $\theta_\xi^G$ transforms via a character under
$Z$~\cite{rogawski:u3}*{Lemma~12.5.1}, it follows that $\pi^s$ and
$\pi^n$ have the same central character.  But $\omega_{\pi^n} =
\lambda_1|_{E^1}\lambda_2$, while $\omega_{\pi^s}$ is the inflation of
$\nu$ to $Z$.  Hence $\nu = \bar{\lambda}_1|_{\{\pm 1\}}\bar{\lambda}_2$.

Let $\gamma\in G_z$ be a regular elliptic element of $G$ whose image
in $\finiteG_z\conn$ is regular.
By~\cite{adler-lansky:bc-u3-unram}*{Lemma 2.1},
the only conjugate of $\stab_G(z)$ that contains $\gamma$
is $\stab_G(z)$ itself.
It follows
from~\cite{adler-lansky:bc-u3-unram}*{Prop.~7.1} that
$\theta_{\pi^s}(\gamma) = \theta_{\bar\sigma}(\bar\gamma)$, where
$\bar\gamma$ is the image of $\gamma$ in $\finiteG_z\conn $.  We now
determine $\bar\sigma$ by evaluating the right-hand side of
(\ref{eq:char_identity}) at appropriate regular elliptic elements of
$G$.

Let $\underline{T}$ be a torus of type~\cartan{3} associated to the
ramified cubic extension $F(\sqrt[3]{\varpi_F})/F$ for some cube root
$\sqrt[3]{\varpi_F}$ of $\varpi_F$ in $\overline F$.  As indicated
in~\S\ref{sec:cubic_torus}, we may choose $\underline{T}$ so that
$T_0\subset G_{\F}$,
where $\F$ is the fundamental alcove in
$\B (\uG ,F)$.  By Lemma~\ref{lem:cubic_torus}, the image of $T_0$ in
$\finiteG_x\conn $ ($x=y$ or
$z$) is the unipotent radical of $\finiteB_x$, whence $T_0$ has
trivial image in $\finiteG_\F$.
Choose $\gamma\in T_0$ with
nontrivial image in $\finiteG_z$.

First note that $\theta_\xi^G (\gamma)$ vanishes since the stable
conjugacy class of $\gamma$ does not meet $H$.  Indeed, $\uH$ has no
maximal $F$-tori of type~\cartan{3}.  We now compute the
value of $\theta_{\pi^n}(\gamma)$, thus determining the value of the
right-hand side of (\ref{eq:char_identity}) at $\gamma$.

According to~\cite{schneider-stuhler:sheaves}*{Lemma~III.4.10,
  Thm.~4.16}, we have that
\begin{equation}
\label{eq:char_pi_n}
\theta_{\pi^n}(\gamma) =
-\tr\left(\gamma\vert(\pi^n)^{G_{\F+}}\right)
+\tr\left(\gamma\vert(\pi^n)^{G_{y+}}\right)
+\tr\left(\gamma\vert(\pi^n)^{G_{z+}}\right).
\end{equation}
We now compute each term of this alternating sum.

Since $\pi^2$ and $\pi^n$ are the constituents of a
principal series representation,
\cite{moy-prasad:k-types} implies that, as representations of
$\finiteG_\F\conn \cong\finiteM\conn $, their spaces of $G_{\F+}$-fixed
vectors are nontrivial sums of one-dimensional characters.  As in the
proof of~\cite{adler-lansky:bc-u3-unram}*{Prop.~4.4}, we can use
Mackey theory and Frobenius reciprocity to see that
$$
\Hom_{G_{\F+}}({\bf 1}, \ind_B^G\lambda )
=\bigoplus_{B\backslash G/G_{\F+}}\C,
$$
so that $|B\backslash G/G_{\F+}|$ is the dimension of the space of
$G_{\F+}$-fixed vectors in $\ind_B^G\lambda$.  Since $y$ is special,
we have that $G=G_yB$.  Therefore, $|B\backslash G/G_{\F+}| =
|(B\cap G_y)\backslash G_y/G_{\F+}| =
|\finiteB_y\backslash\finiteG_y/\finiteB_y|=2$, by the Bruhat
decomposition.  Thus the space of $G_{\F+}$-fixed vectors in
$\ind_B^G\lambda$ is two dimensional,
so $(\pi^2)^{G_{\F+}}$ and
$(\pi^n)^{G_{\F+}}$ are both one dimensional.  In particular, since the
image of $\gamma$ in $\finiteG_{\F}\conn $ is trivial, we have
\begin{equation}
\label{eq:trace_F}
\text{trace}\left(\gamma|(\pi^n)^{G_{\F+}}\right) = 1.
\end{equation}

An argument contained in the proof
of~\cite{adler-lansky:bc-u3-unram}*{Prop.~4.4} shows that the
representation $\pi^{G_{y+}}$ of $G_y/G_{y+} =
\finiteG_y\conn $ is equivalent to
$\ind_{\finiteB_y}^{\finiteG_y\conn }\bar\lambda_1$.  Since
$\bar\lambda_1 = \bar\eta$ is the character of
$\finiteM$ of order $2$, it follows from
\S\ref{sub:reductive-quotients} that the
latter representation decomposes as $\phi\oplus\St_{\finiteG_y\conn }
\cdot\phi$, where $\phi$ is the character of
$\finiteG_y\conn $ of order $2$.  As in the proof
of~\cite{adler-lansky:bc-u3-unram}*{Prop.~4.4}, it follows that
$(\pi^n)^{G_{y+}}\cong\phi$, while
$(\pi^2)^{G_{y+}}\cong\St_{\finiteG_y\conn }\cdot\phi$.  In particular,
since the image $\bar\gamma$ of $\gamma$ in $\finiteG_y$ is unipotent,
\begin{equation}
\label{eq:trace_y}
\tr\left(\gamma|(\pi^n)^{G_{y+}}\right) = \phi(\bar\gamma) = 1.
\end{equation}

As in the case
of $G_y$, the representation $\pi^{G_{z+}}$ of $\finiteG_z\conn $ is
equivalent to $\ind_{\finiteB_z}^{\finiteG_z\conn }\bar\lambda_1$.  Since
$\bar\lambda_1$ has order $2$,
it follows from \S\ref{sub:reductive-quotients} that the latter
representation decomposes as $\vartheta\oplus\vartheta'$, where
$\vartheta$ and $\vartheta'$ are the two inequivalent irreducible
representations of $\finiteG_z\conn $ of degree $(q+1)/2$.
Again, 
$(\pi^n)^{G_{z+}}$ must be equivalent to one of these
representations, say $\vartheta$.  This proves (i).  In particular, we
have
\begin{equation}
\label{eq:trace_z}
\tr\left(\gamma|(\pi^n)^{G_{z+}}\right) = \theta_{\vartheta}(\bar\gamma),
\end{equation}
where $\bar\gamma$ is the image of $\gamma$ in $\finiteG_z\conn $.

Therefore, by \eqref{eq:char_pi_n},
\eqref{eq:trace_F}, \eqref{eq:trace_y}, and \eqref{eq:trace_z}, we
have
$$
\theta_{\pi^n}(\gamma) = -1+1+\theta_{\vartheta}(\bar\gamma) =
\theta_{\vartheta}(\bar\gamma).
$$
Hence \eqref{eq:char_identity} implies that 
$$
\theta_{\bar\sigma}(\bar\gamma) = \theta_{\pi^s}(\gamma) =
 - \theta_{\vartheta}(\bar\gamma).
$$
This completes the proof of (ii).
\end{proof}

Since Proposition~\ref{prop:A-packet} determines the relationship
between $\pi^s$ and $\pi^n$, hence between $\pi^s$ and $\pi^2$, it is
sufficient for the verification of this case of Theorem~\ref{thm:main}
in \S\ref{sec:base_change}.
In particular, we do not need to know precisely
how these representations depend on the character $\lambda$.
For
completeness, we provide such a description anyway in the following
result, which determines $\vartheta$ and hence $\bar\sigma$ explicitly
in terms of the inducing character $\lambda$.  It is enough to compute
the values of $\vartheta$ and $\bar\sigma$ on any nontrivial unipotent
element of $\finiteG_z\conn $, as these values uniquely determine
these representations (see~\cite{fulton-harris:rep-theory}*{\S5.2}).
The proof involves a careful analysis of intertwining operators
on the representations $\pi = \ind_B^G \lambda$ of $G$ and
$\ind_{\finiteB_z}^{\finiteG_z\conn }\bar\eta$ of $\finiteG_z\conn $.

It follows from the discussion
in \S\ref{sub:reductive-quotients} that the group $\finiteG_z\conn $
may be explicitly realized as the group of matrices of the form
$$
\begin{pmatrix}
        a & 0 & b \\
        0 & 1 & 0  \\
        c & 0 & d
\end{pmatrix}
\in \widetilde{\finiteG}_z\conn  = \GL(2,k)\times\GL(1,k)
$$
such that 
$\det\smattwo{a}{b}{c}{d} =1$.
Moreover, $\finiteB_z$ is identified
with the group of lower-triangular matrices in $\finiteG_z\conn $ in
this realization.  Choose an element $\gamma$ of the cubic torus $T$
from the proof of the preceding proposition such that  
$\bar\gamma = 
\left(  \begin{smallmatrix}
        1 & 0 & 0 \\
        0 & 1 & 0  \\
        1 & 0 & 1
        \end{smallmatrix}
        \right)
\in\finiteB_z\subset\finiteG_z\conn$.

\begin{prop}
\label{prop:A-packet_explicit}
Let $\lambda $ be a depth-zero character of $M$ such that
$\lambda_1 = \eta |\cdot |_E^{\pm 1/2}$,
where $\eta|_{F\mult} = \omega_{E/F}$.
Then the representation $\vartheta = \pi^n(\lambda)^{G_{z+}}$ of
$\finiteG_z\conn $ satisfies $\theta_{\vartheta} (\bar\gamma) = 
(1-\eta(2\varpi_E)q^{1/2})/2$.
Thus the cuspidal representation $\bar\sigma$
of $\finiteG_z\conn $ from Proposition~\ref{prop:A-packet}(ii)
satisfies $\theta_{\bar\sigma} (\bar\gamma) = 
(-1+\eta(2\varpi_E)q^{1/2})/2$.
\end{prop}

\section{Explicit base change for ramified unitary groups in three variables}
\label{sec:base_change}

\subsection{$L$-packets consisting of principal series constituents}
\label{sub:bc_ps}
The following proposition gives the base change lifts of the principal
series $L$-packets of $G$.  Recall that
$\utH$ is an $F$-Levi subgroup of $\utP$.  Note that there is a natural
isomorphism $\utH\longrightarrow\utH^0\times\utZ$.  In particular,
$\tH\cong\GL(2,E)\times\GL(1,E)$, and in the
statement of the proposition, we identify these groups.
\begin{prop}
\label{prop:bc_ps}
Let $\lambda$ be a character of $M$ and let $\tilde\lambda$ denote the
character of $\tM$ lifted from $\lambda$.  Let $\lambda_1$ and
$\lambda_2$ be the characters of $E^\times$ and $E^1$, respectively,
associated to $\lambda$.  Let $\tilde\lambda_2$ be the character
$\lambda_2\circ\N^{U(1)}|_{E\mult}$ of $E\mult$.
\begin{enumerate}[(i)]
\item If $\ind_B^G\lambda$ is irreducible and
        $\ind_{\widetilde B}^{\widetilde G}\tilde\lambda$
   is irreducible, then the base
  change lift of the $L$-packet $\{ \ind_B^G \lambda\}$ is
  $\ind_{\widetilde{B}}^\tG\tilde\lambda$.
\item  If $\ind_B^G\lambda$ is irreducible but
   $\ind_{\widetilde B}^{\widetilde G}\tilde\lambda$ is reducible, then
  $\lambda_1|_{F\mult} = |\phantom{x}|_F^{\pm1}$, and the base
  change lift of $\{ \ind_B^G \lambda\}$ is the constituent
$$
\ind_{\widetilde
    P}^{\widetilde G} \Bigl( ( \lambda_1\tilde\lambda_2|\cdot|_E^{\mp
    1/2}\circ \det_{\utH^0}) \otimes \tilde\lambda_2 \Bigr).
$$
of $\ind_{\widetilde{B}}^\tG\tilde\lambda$.
\item  If $\lambda_1 = |\phantom{x} |_E^{\pm 1}$, then the lift
  of the $L$-packet consisting of the one-dimensional constituent
  $\psi=\lambda_2\circ\det_{\uG}$ (respectively, the
  constituent $\St_G\cdot\psi$) of $\ind_B^G\lambda$ is the
  one-dimensional constituent
$\tilde{\psi} = \tilde\lambda_2\circ\det_\utG$
(respectively, the constituent
$\St_\tG\cdot\tilde{\psi}$) of
  $\ind_{\widetilde{B}}^\tG \tilde\lambda$.
\item If $\lambda_1|_{F\mult}$ is trivial and $\lambda_1$ is
  nontrivial, then the lift of the $L$-packet
  $\{\pi_1(\lambda),\pi_2(\lambda)\}$ is
  $\ind_{\widetilde{B}}^\tG \tilde\lambda$.
\end{enumerate}
\end{prop}

\begin{proof}
The statement and proof of this proposition are identical to
those of~\cite{adler-lansky:bc-u3-unram}*{Prop.~4.1}.
\end{proof}

\subsection{Supercuspidal $L$-packets}
\label{sub:bc_sc}
Let $\rho$ be an irreducible admissible representation of $H =
H^0\times Z$.
Suppose $\rho$ factors as $\rho^0\otimes\psi$ for some
irreducible admissible representations $\rho^0$ of $H^0$ and $\psi$ of
$Z$.
The group $\uH = \uH^0\times \uZ$
may also be factored as $\uH = \uH^0\times \uJ$,
where $\uJ\cong\uZ\cong\UU(1,E/F)$ is the $F$-subgroup of $\uH$ consisting
of matrices of the form
$
\left(
\begin{smallmatrix}
1 & 0 & 1 \\
0 & * & 0 \\
1 & 0 & 1
\end{smallmatrix}
\right).
$
The factorization of $\rho$ with respect to the latter decomposition of $H$
(denoted by the symbol $\otimes'$) is 
$$
\rho^0\otimes '\psi\omega\inv_{\rho^0},
$$
where we view $\psi$ and
$\omega_{\rho^0}$ as representations of $J$ via the obvious
isomorphism of $J$ with $Z$.
An analogous
result holds for $\tH$ when one converts between the decompositions
$\utH^0\times\utZ$ and $\utH^0\times\utJ$ of $\utH$, where $\utJ$ is
the $F$-subgroup of $\utH$ consisting of matrices of the form
$
\left(
\begin{smallmatrix}
1 & 0 & 1 \\
0 & * & 0 \\
1 & 0 & 1
\end{smallmatrix}
\right).
$
We will again use the notation $\otimes'$ to denote factorization with
respect to the second decomposition.

We now recall the labile base change lifting for the groups $H^0$
and $H$ (see~\cite{rogawski:u3}*{\S11.4,12.1}).
If $\rho^0$ is an element of a singleton $L$-packet, let
$\tilde\rho^0$ denote the (stable) base change lift of $\rho^0$ to
$\tH^0$.
The labile base change lift $\rho^0$ is defined to be
$\tilde\rho^0\cdot(\Omega\circ\det_{\utH^0})$.
Similarly, the labile base
change lift of the irreducible representation $\rho\otimes'\psi$ of
$H = H^0\times J$ is defined to be the representation
$\tilde\rho^0\cdot(\Omega\circ\det_{\utH^0})\otimes'\tilde\psi$ of
$\tH^0\times\widetilde J$.  Here, as usual, $\tilde\psi$ denotes the
character of the torus $\widetilde J$ lifted from $\psi$.

We will need the following lemma to determine the base change lift of
a supercuspidal $L$-packet of $G$.  For the remainder of this section,
$\det$ will be understood to mean $\det_{\utH^0}$.
\begin{lem}
\label{lem:bc-labile}
Let $\rho = \rho^0\otimes\psi$ be an irreducible admissible
representation of $H = H^0\times Z$.  If $\{\rho^0\}$ is an $L$-packet
of $H^0$,
then the labile base change lift $\tilde\rho'$
of $\rho$ is 
$$
\tilde\rho^0\cdot (\Omega\circ\det)\otimes\tilde\psi\Omega^2.
$$
\end{lem}
\begin{proof}
  By the preceding discussion, we may factor $\rho$ as $\rho^0\otimes'
  \psi\omega\inv_{\rho^0}$ with respect to the decomposition
  $H=H^0\times J$.
  Then $\tilde\rho'$ is equal to
  $$
  \tilde\rho^0\cdot(\Omega\circ\det) \otimes'
  \tilde\psi\tilde\omega\inv_{\rho^0}.
  $$
  Thus, with respect to the
  decomposition $\tH = \tH^0\times\widetilde{Z}$, $\tilde\rho'$
  factors as
  $$
  \tilde\rho^0\cdot(\Omega\circ\det) \otimes
  \tilde\psi\tilde\omega\inv_{\rho^0}
  \omega_{\tilde\rho^0\cdot(\Omega\circ\det)}.
  $$
  But 
  $$
  \omega_{\tilde\rho^0\cdot(\Omega\circ\det)} =
  \tilde\omega_{\rho^0}\Omega^2,
  $$
  where, on the we right side, we view $\Omega$ as a character of
  $\widetilde{Z} = E\mult$.
  The result follows.
\end{proof}

We now consider a supercuspidal $L$-packet of $G$.
Fix a character
$\nu$ of $\finiteZ$ and a character
$\chi$ of $\finiteS^\circ$ of order greater
than $2$. 
We recall the notation of~\S\ref{sub:sc}.
As in \S\ref{sec:depth-zero-G}, $\chi$ and $\nu$ determine a cuspidal
representation $\bar\sigma_x$ of $\finiteG_x$,
where $x$ is either $y$ or $z$.
Similarly, as in the discussion before
Proposition~\ref{prop:sc_card2}, $\chi\phi$ and $\nu\bar\Omega$
determine a cuspidal
representation $\bar\tau$ of
$\finiteH_z$, where $\phi$ is the
character of $\finiteS^0$ of order $2$.
Let 
$$
\pi_x = \ind_{\stab_G(x)}^G\sigma_x, \qquad
\rho = \ind_{\stab_H(z)}^H\tau,
$$
where $\sigma_x$ and $\tau$ are the inflations of $\bar\sigma_x$ and
$\bar\tau$, respectively. 
Then by Proposition~\ref{prop:sc_card2}, the $L$-packet
$\Pi = \{\pi_y,\pi_z\}$ is the transfer via $\Xi_G$ of $\{\rho\}$ from $H$.
Note that we can express $\rho$ as the representation
$\rho^0\otimes\widehat\nu\Omega$ (relative to the decomposition
$H=H^0\times Z$), where
$\rho^0$ is the representation of $H^0$ of the form~\Hdepthzerosc{1}
coming from the cuspidal Deligne-Lusztig representation $\bar\sigma$
of $\finiteH^0$ associated to $\chi\phi$.

\begin{prop}
\label{prop:bc_sc}
The base change lift of the $L$-packet $\Pi = \{\pi_y,\pi_z\}$ is the
representation
$\ind_{\widetilde{P}}^\tG
\left(
\tilde\rho^0\cdot(\Omega\circ\det)\otimes \widetilde{\widehat{\nu}}
\right)$,
where $\tilde\rho^0$ is the base change
lift of $\rho^0$ to $\tH^0$.
\end{prop}
It is readily checked that the equivalence class of
$\tilde\rho^0\cdot(\Omega\circ\det)$, and hence that of 
$\ind_{\widetilde{P}}^\tG
\left(
\tilde\rho^0\cdot(\Omega\circ\det)\otimes \widetilde{\widehat{\nu}}
\right)$,
is independent of the choice of $\Omega$.
\begin{proof}
Since $\Pi$ is the endoscopic transfer of $\rho$,
  \cite{rogawski:u3}*{Prop.~13.2.2(c)}
  implies that the base change lift
  of $\Pi$ is $\ind_{\widetilde{P}}^\tG(\tilde\rho')$,
  where $\tilde\rho'$ is the labile base change lift of $\rho$.
  It follows from
  Lemma~\ref{lem:bc-labile} that the labile base change lift of $\rho
  = \rho^0\otimes\widehat\nu\Omega$ is
  $$
  \tilde\rho^0\cdot (\Omega\circ\det)\otimes
  \widetilde{\widehat\nu}\widetilde\Omega\Omega^2.
  $$
  Now as a character of $E\mult$, $\widetilde\Omega$ is the map
  $a\mapsto \Omega (a/\iota(a))$.
  Since $\Omega|_{F\mult} = \omega_{E/F}$, it follows that for any
  $a\in E^\times$, $\widetilde\Omega\Omega^2 (a) =
  \widetilde\Omega\inv\Omega^2 (a) = 
  \Omega (N_{E/F}(a)) = 1$.
  This proves the proposition.
\end{proof}

\subsection{Non-supercuspidal $L$-packets containing supercuspidals}
\label{sub:stable_sc}

\begin{prop}
\label{prop:bc_card_2}
Let $\lambda$ be a character of $M$ of depth zero such that
$\lambda_1|_{F\mult} = \omega_{E/F}|\cdot |_F^{\pm 1}$.
\begin{enumerate}[(i)]
\item The base change lift of the $L$-packet
$\{ \pi^2 (\lambda),\pi^s (\lambda)\}$ is 
$$
\ind_{\widetilde{P}}^\tG \left(\St_{\widetilde{H}}\cdot\left( 
(\lambda_1 \tilde{\lambda}_2 |\cdot |_E^{\mp 1/2} \circ\det_{\utH^0})
\otimes \tilde{\lambda}_2  
\right)\right).
$$
\item The base change lift of the $A$-packet
$\{ \pi^n (\lambda),\pi^s (\lambda)\}$ is 
$$
\ind_{\widetilde{P}}^\tG \left( 
(\lambda_1 \tilde{\lambda}_2 |\cdot |_E^{\mp 1/2} \circ\det_{\utH^0} )
\otimes \tilde{\lambda}_2  
\right).
$$
\end{enumerate}
Moreover, the above two base change lifts are precisely the
irreducible constituents of the principal series representation
$\ind_\tB^\tG \tilde\lambda$.
\end{prop}

Note that the proposition has the same content if we restrict
the choice of exponent in the hypothesis to be $+1$
(or to be $-1$).

\begin{proof}
Identical to the proof of \cite{adler-lansky:bc-u3-unram}*{Prop.~4.2}.
\end{proof}

\section{Compatibility of base change and minimal $K$-types}
\label{sec:K-types}

\subsection{Principal series $L$-packets}
\label{subsub:principal_series}
As in~\S\ref{sub:bc_ps}, suppose $\Pi$ consists entirely of
constituents of the depth-zero principal series representation
$\ind_B^G \lambda$.
Since each element of $\Pi$ has depth zero, $\ind_B^G \lambda $ and
hence $\lambda$ have depth zero
by~\cite{moy-prasad:jacquet}*{Theorem~5.2}.
It follows from~\cite{moy-prasad:jacquet} that
$(G_\F,\infl\bar\lambda)$
is a minimal $K$-type contained in each element of $\Pi$, where
we have identified $\finiteM\conn$ with $G_\F/G_{\F+}$.
Similarly,
$(\tG_\F, \infl\tilde{\bar\lambda})$ is a minimal $K$-type of
each constituent of
$\tilde{\pi} = \ind_{\widetilde{P}}^\tG \tilde\lambda$ (see
Proposition~\ref{prop:bc_ps}), where
$\tG_{\F}/\tG_{\F+}$ is identified with
$\widetilde{\finiteM}$.
But the image of $\bar\lambda$ under the $\varepsilon$-lifting
from $\ufM\conn$ to
$\utfM$ of  is $\tilde{\bar\lambda}$.  Thus the Main Theorem
holds in this case.

\subsection{Supercuspidal $L$-packets}
Recalling the notation of \S\ref{sub:bc_sc}, suppose that $\Pi$ is the
supercuspidal $L$-packet $\{\pi_y(\chi,\nu),\pi_z(\chi,\nu)\}$.
For $x=y$ or $z$, $\pi_x$ contains the minimal $K$-type $(G_x,
\sigma_x)$, where $\bar\sigma_x$ is the cuspidal Deligne-Lusztig
representation of $\finiteG_x$ associated to the character $\chi$ of
$\finiteS^0 = \finiteS\conn \cong (\finiteS')\conn$.

By Proposition~\ref{prop:bc_sc}, the base change lift $\tilde\pi$ of
$\Pi$ is $\ind_{\widetilde{P}}^\tG
\left(
\tilde\rho^0\cdot(\Omega\circ\det)\otimes \widetilde{\widehat{\nu}}
\right)$, where $\rho^0$ is the supercuspidal representation of $H^0$
coming from the cuspidal Deligne-Lusztig representation $\bar\sigma$
of $\finiteH^0$ associated to $\chi\phi$.  
Therefore, by
Proposition~\ref{prop:bc_H_sc}, $\tilde\rho^0\cdot(\Omega\circ\det)$
is the representation
\begin{eqnarray*}
\lefteqn{
\Bigl(
        \ind_{\tH^0_z\tZ^0}^{\tH^0}\ext \infl
        \bigl(
                R_{\widetilde\ufS^0}^{\widetilde\ufH_z^0}\widetilde{\chi\phi}
        \bigr)
\Bigr) 
\cdot(\Omega\circ\det)}\\
& = & \ind_{\tH^0_z\tZ^0}^{\tH^0}\Bigl(\ext \infl
\bigl( R_{\widetilde\ufS^0}^{\widetilde\ufH_z^0}(\tilde\chi\tilde\phi)\bigr)
\cdot(\Omega\circ\det) \Bigr)  .
\end{eqnarray*}
Here the notation ``$\infl$'' denotes the inflation from
$\widetilde\finiteH_z^0$ to $\widetilde H_z^0$, and ``$\ext$'' denotes
the unique extension of the given representation of $\tH_z^0$ to
$\tH_z^0\tZ$ such that the
element of $\tZ$ corresponding to $\varpi_E$ acts via $(\chi\phi)(-1)$.
Now
$$
\ext \infl
\bigl( R_{\widetilde\ufS^0}^{\widetilde\ufH_z^0}(\tilde\chi\tilde\phi)\bigr)
\cdot(\Omega\circ\det)\\
=\ext'
\infl
\Bigl( R_{\widetilde\ufS^0}^{\widetilde\ufH_z^0}(\tilde\chi\tilde\phi)
\cdot(\bar\Omega\circ\det)
\Bigr) ,
$$
where ``$\ext'$'' denotes the extension such
that $\varpi_E$ acts via $(\chi\phi)(-1)\Omega(\varpi_E^2) =
(\chi\phi)(-1)\Omega(-1) = -\chi(-1)$.
Since the restriction of $\bar\Omega\circ\det$ to $\utfS^0$ is
quadratic, the latter representation is equivalent to 
$\ext' \infl
\bigl( R_{\widetilde\ufS^0}^{\widetilde\ufH_z^0}(\tilde\chi)\bigr)$.
Thus
$\tilde\rho^0\cdot(\Omega\circ\det)\otimes\widetilde{\widehat{\nu}}$
is equivalent to 
$$
\ind_{\tH^0_z\tZ}^{\tH}
\Bigl(
\ext' \infl
\bigl(
        R_{\widetilde\ufS^0}^{\widetilde\ufH^0_z}\tilde\chi\bigr) 
\otimes \widetilde{\widehat{\nu}}
\Bigr) .
$$
Since $\widetilde{\widehat{\nu}}$ is trivial on
the maximal compact subgroup $\tZ_0$ of $\tZ$,
the representation 
$$
\ext' \infl
\bigl(
        R_{\widetilde\ufS^0}^{\widetilde\ufH^0_z}\tilde\chi\bigr) 
\otimes \widetilde{\widehat{\nu}}
$$
of $\tH^0_z\tZ$ can be expressed as an extension to $\tH^0_z\tZ$ of
the representation
$$
\infl
\bigl(
        R_{\widetilde\ufS^0}^{\widetilde\ufH^0_z}\tilde\chi\bigr)\otimes
{\bf 1}_{\tZ_0} = 
\infl
        R_{\widetilde\ufS}^{\widetilde\ufH_z}(\tilde\chi\otimes {\bf
          1}_{\tfZ}).
$$
Since $\tilde\pi$ is parabolically induced from the representation 
$\tilde\rho^0\cdot(\Omega\circ\det)\otimes\widetilde{\widehat{\nu}}$,
which contains
$$
\Bigl(
\tH_z,
\infl
(R_{\widetilde\ufS}^{\widetilde\ufH_z}(\tilde\chi\otimes {\bf
          1}_{\tfZ}))\Bigr),
$$
and since $\tfG_z \cong \tfH_z$, we have from
\cite{moy-prasad:jacquet}*{Thm.~6.11} that
$\tilde\pi$ contains the $K$-type 
\begin{equation}
\label{eq:datum}
\Bigl(\tG_z,\infl\bigl(
R_{\widetilde{\ufS}}^{\widetilde{\ufG}_z}\bigl(\tilde\chi
\otimes {\bf 1}_{\tfZ}\bigr) \bigr)\Bigr) .
\end{equation}
But $R_{\widetilde{\ufS}}^{\widetilde{\ufG}_z}\bigl(\tilde\chi
\otimes {\bf 1}_{\tfZ}\bigr)$
is the $\varepsilon$-lift of $\bar\sigma_z$ since the character $\tilde\chi
\otimes {\bf 1}_{\tfZ}$ of $\widetilde{\finiteS}$ is the $\varepsilon$-lift of
the character $\chi$ of $\finiteS$.

We now consider the vertex $y$.  Note that 
$$
R_{\widetilde{\ufS'}}^{\widetilde{\ufG}_y}
\bigl(\tilde\chi\otimes {\bf 1}_{\tfZ} \bigr)
\cong \ind_{\finiteP_{z'}}^{\widetilde{\finiteG}_y} \bigl(
  R_{\widetilde{\ufS'}}^{\widetilde{\ufG}_{z'}}
\bigl(\tilde\chi\otimes {\bf 1}_{\tfZ}\bigr)
\bigr) ,
$$
where $S'=gSg\inv$ is the stable conjugate of $S$ fixing $y$
(see~\ref{sec:quadratic_torus}), $z' = g\cdot z$, and $\finiteP_{z'}$
is the image of $\tG_{z'}$ in $\widetilde{\finiteG}_y$.
Now $\tilde\pi$ must contain the minimal $K$-type
$$
\Bigl(\tG_{z'},
\infl\bigl(
R_{\widetilde{\ufS}}^{\widetilde{\ufG}_{z'}}\left(\tilde\chi
\otimes {\bf 1}_{\tfZ}\right) \bigr)\Bigr) ,
$$
since it is an associate of \eqref{eq:datum}.  
An application of Frobenius reciprocity then shows that $\tilde\pi$
must also contain
$$
\Bigl(\tG_y,
\infl\bigl(
R_{\widetilde{\ufS}'}^{\widetilde{\ufG}_y}\bigl(\tilde\chi
\otimes {\bf 1}_{\tfZ}\bigr) \bigr)\Bigr) ,
$$
and $R_{\widetilde{\ufS}'}^{\widetilde{\ufG}_y}(\tilde\chi
\otimes {\bf 1}) $
is the $\varepsilon$-lift of $\bar\sigma_y$.

\subsection{Non-supercuspidal $L$-packets containing supercuspidals}
Now suppose $\Pi $ is an $L$-packet of the form
$\{ \pi^2(\lambda) ,\pi^s(\lambda)\}$
or an $A$-packet of the form
$\{ \pi^n(\lambda) ,\pi^s(\lambda)\}$
for some $\lambda\in\Hom (M,\C\mult)$ of depth
zero (see case (\ref{sec:depth-zero-G}PS--2) and \S\ref{sub:ps}).
In either case, let $\tilde\pi$ denote the base change lift of $\Pi$.
Both $\pi^2(\lambda)$ and $\pi^n(\lambda)$ are constituents of the
principal series representation $\ind_B^G\lambda$.
It follows
from~\cite{moy-prasad:jacquet} that $\ind_B^G\lambda$ has
depth zero and that, for any $x\in\F$, $(G_x,\infl\bar\lambda)$ is a minimal
$K$-type for both $\pi^2(\lambda)$ and $\pi^n(\lambda)$.
By Proposition~\ref{prop:bc_card_2}, $\tilde{\pi}$ is a
constituent of the principal series
$\ind_{\widetilde{B}}^\tG \tilde\lambda$.
Therefore, as in \S\ref{subsub:principal_series},
$(\tG_{x},\infl\tilde{\bar\lambda})$ is a
minimal $K$-type for $\tilde{\pi}$.
The Main Theorem now follows for $\pi^2(\lambda)$ and $\pi^n(\lambda)$
exactly as it did in~\ref{subsub:principal_series}.

It remains to consider $\pi^s(\lambda)$
(both as an element of
$\{\pi^s(\lambda),\pi^2(\lambda)\}$ and as one of
$\{\pi^s(\lambda),\pi^n(\lambda)\}$).
As shown in Proposition~\ref{prop:A-packet}, $\pi^s(\lambda)$ is a
supercuspidal representation of the form~\depthzerosc{3}.  Therefore,
$\pi^s(\lambda)$ contains a minimal $K$-type of the form
$(G_z,\infl(\bar\sigma\otimes\nu)$, where $\bar\sigma$ is a cuspidal
representation of $\finiteG_z^\circ$ of degree $(q-1)/2$ and $\nu$ is a
character of $\finiteZ$.  Therefore, $(G_z,\infl\bar\sigma)$ is a
minimal $K$-type of $\pi^s(\lambda)$.

Now $\bar\sigma$ is a subrepresentation of
$R_{\ufS\conn}^{\ufG_z\conn}\phi$, where $\phi$ is the nontrivial
quadratic character of $\finiteS\conn$.  Hence $\bar\sigma$ is a member
of the rational series of representations of $\finiteG_z\conn$ attached to
$\phi$ (see \S\ref{sub:lifting-finite}).  The image of this rational series
under the $\varepsilon$-lifting is
the rational series of representations of $\tfG_z$ attached to the character
$\tilde\phi$ of $\tfS$.  Since $\tilde\phi$ is the nontrivial
quadratic character $\bar\eta$ of $\tfS$, it extends to $\tfG_z$, and it is
easily seen that the latter series consists of the representations
$\bar\eta$ and $\St_{\tfG_z}\cdot\,\bar\eta$.

Since the base change lift of the $L$-
or $A$-packet $\Pi$ is an irreducible constituent of 
$\ind_\tB^\tG\tilde\lambda$,
we have (as in the proof of
Proposition~\ref{prop:A-packet}) 
$$
\left(\ind_\tB^\tG\tilde\lambda\right)^{\tG_z+}\cong
\ind_{\tfB_z}^{\tfG_z}\tilde{\bar\lambda},
$$
so $\tilde\pi^{\tG_z+}$ is a subrepresentation of 
$\ind_{\tfB_z}^{\tfG_z}\tilde{\bar\lambda}$.

Since $\lambda_1|_{F\mult} = \omega_{E/F}|\cdot|_F^{\pm
  1}$, one readily computes that the character $\tilde{\bar\lambda}$ of
$\tfM$ is equal to
$(\bar\eta\circ\det)\otimes{\bf 1}$, where the factorization is with
  respect to the decomposition $\tfM = \tfM^0 \times \tfZ$.  Thus 
$$
\ind_{\tfB_z}^{\tfG_z}\tilde{\bar\lambda}\cong \bar\eta\oplus
\left(\St_{\tfG_z}\cdot\,\bar\eta\right).
$$
Hence $\tilde\pi$ contains a representation in the $\varepsilon$-lift
of the rational series containing $\bar\sigma$.

\end{document}